\documentclass[11pt,reqno]{amsart}

\usepackage{amssymb}
\usepackage{amsthm}
\usepackage{amsmath}
\usepackage{graphicx}
\usepackage{amsaddr}
\usepackage{comment}
\usepackage[usenames,dvipsnames]{xcolor}
\usepackage[margin = 1in]{geometry}
\usepackage{enumerate}
\usepackage[colorlinks=true, allcolors=blue]{hyperref}
\usepackage[toc,page]{appendix}
\usepackage{lineno}
\usepackage{color}
\usepackage{bbm}
\usepackage{hyperref}
\usepackage{mhchem}
\usepackage{siunitx}

\newcommand{\blue}{\color{blue}}
\definecolor{mygreen}{rgb}{0.1,0.75,0.2}

\providecommand{\bbs}[1]{\left(#1\right)}

 \newtheorem{thm}{Theorem}[section]
 \newtheorem{cor}[thm]{Corollary}
 \newtheorem{lem}[thm]{Lemma}
 \newtheorem{prop}[thm]{Proposition}

 \newtheorem{rem}[thm]{Remark}

 \numberwithin{equation}{section}

\providecommand{\bbs}[1]{\left(#1\right)}
\newcommand{\lra}{\longrightarrow}

\newcommand{\wra}{\rightharpoonup}

\newcommand{\la}{\langle}
\newcommand{\ra}{\rangle}
\newcommand{\pt}{\partial}
\newcommand{\eps}{\varepsilon}
\newcommand{\ud}{\,\mathrm{d}}
\newcommand{\8}{\infty}

\newcommand{\bR}{\mathbb{R}}

\newcommand{\sP}{\mathcal{P}}

\newcommand{\rhoe}{\rho_t^\eps}
\newcommand{\pie}{\pi_\eps}
\newcommand{\ve}{{V_\eps}}
\newcommand{\pio}{\pi_0}
\newcommand{\be}{B_\eps^{-1}}

\newcommand{\lbar}{\overline}

\begin{document}

\title
[Homogenization of Wasserstein gradient flow for McKean-Vlasov Equation]
{Homogenization and Long Time Behavior of Wasserstein gradient flow for McKean-Vlasov Equation}

\author[Y. Gao]{Yuan Gao}
\address{Department of Mathematics, Purdue University, West Lafayette, 47906}
\email{gao662@purdue.edu}

\author[N. K. Yip]{Nung Kwan Yip}
\address{Department of Mathematics, Purdue University, West Lafayette, 47906}
\email{yipn@purdue.edu (corresponding author)}

\begin{abstract}
We consider a McKean-Vlasov equation incorporating spatial fast oscillatory behavior into both
the energetic and kinematic components of the model. The former is given in terms of a confining potential while
the latter by an underlying metric. The equation is formulated as a gradient flow in the space of probability measures endowed with a Wasserstein metric. We identify the limiting dynamics
which is also described as a gradient flow in an effective media. We further prove the long time exponential convergence of solutions to the invariant measure. The rate of convergence is expressed in terms of a Logarithmic Sobolev constant which is uniform in terms of the oscillatory length scale. Our approach is based on variational
and functional inequalities. The homogenization result of the current paper extends the authors' previous work \cite{GaoYip} on a linear Fokker-Planck equation to a nonlinear and nonlocal setting.
\end{abstract}

\date{\today}

\maketitle

\section{Introduction}

Systems of interacting particles subject to independent random forces arise
throughout the physical, biological and social sciences, modeling phenomena ranging from
plasma dynamics \cite{vlasov1961many}, 
granular media \cite{Villani2006Granular},
swarming and flocking \cite{cucker2007emergent,Carrillo2010Swarming,carrillo2014derivation}, 
to opinion formation \cite{deffuant2000mixing,hegselmann2002opinion}.
When the number of particles is large and the interaction is of \emph{mean-field type} ---
each particle interacts with all the other particles in some appropriate average
sense,
the empirical measure of the system is well approximated by the solution of a
McKean--Vlasov equation \cite{McKean1966, McKean1967} which is a nonlinear and nonlocal Fokker--Planck equation of
aggregation--diffusion type \eqref{mainFP}.
There has been a large number of work on this equation, 
starting from the derivation of continuum equations from particle systems to 
analysis of well-posedness and long time behavior. We refer to
\cite{golse2016dynamics,jabin2017mean} for some overviews of this
topic. 

In this work, we will make use of variational and gradient flow structure of the equation. This approach has shown to be very fruitful as it can tackle a wide range of nonlinear equations and provide powerful tools in the study of well-posedness, stability, and convergence to equilibrium of the solutions. To mathematically implement such an
approach, we will mention two key ideas in connection to our current paper. 
The first is the work \cite{otto2001geometry} by Otto which
introduces a Riemannian metric structure for the space of probability measures $\sP$
which is then used to
analyze long time behavior of porous medium equation. Another work \cite{jordan1998variational}, by Jordan, Kinderlehrer and Otto re-writes a Fokker-Planck equation as a gradient flow with respect to
some Wasserstein distance $W$ for $\sP$. Then solutions are constructed by means of a time discretization variational scheme. Using such a framework, general aggregation-diffusion type equations 
can be viewed as gradient flows of a free energy (see \eqref{E_mc_eps}) --- consisting of entropy, a potential energy, and an interaction energy --- with respect to $W$
\cite{carrillo2019aggregation}.
The second idea is initiated by the work \cite{sandier2004gamma} of Sandier and Serfaty which gives
an equivalent formulation of gradient flows in terms of a variational inequality. 
This approach enables the use of techniques from calculus of variations to analyze the convergence and asymptotics of many gradient flow dynamics. In some sense, this extends
the idea of $\Gamma$-convergence for stationary problems to their evolutionary counterparts --- see \cite{mielke2016evolutionary,AGS} for more recent overviews of this approach.
Both of these two ideas will be explained in more detail later in this introductory section and 
also Section \ref{sec2.2}.
We refer to \cite{CarrilloMcCannVillani2003, CarrilloMcCannVillani2006, carrillo2019aggregation} for further expositions and applications in nonlinear partial differential equations, in particular the first in the list which describes
a connection between long time behavior of the solution and the underlying geometry
of Wasserstein space.

In many applications, the environment in which the particles evolve is
far from homogeneous. This can happen in 
transportation networks \cite{bernot2008optimal},
population dynamics \cite{tilman1997spatial},
protein folding \cite{onuchic1997theory} and many others. 
See also \cite{pavliotis2008multiscale,Du2020Multiscale} and references herein. 
The underlying energy landscape may be rough, fluctuating rapidly on a small spatial scale $\eps\ll1$, and the medium itself may also 
inhibit or facilitate transport in a spatially varying and anisotropic manner. 
Both of these features are naturally captured by incorporating confining potential 
$U_\eps$ and the mobility matrix $\be$ which oscillate on scale $\eps$. 
Loosely speaking, the energetics of the model (described by $U_\eps$) interacts
in some interesting and
non-trivial manner with the kinematics (described by $\be$).

Two questions then arise. The first is a question of \emph{homogenization}: as $\eps \to 0$,
does the dynamics converge, and if so, is the limit again a gradient flow of an
effective energy with respect to an effective transport metric? The second
concerns the \emph{long-time behavior of the oscillatory system itself}: for
each fixed $\eps$, does the solution converge to its equilibrium, and at what
rate? Since in practice $\eps$ is small but positive, it is essential that the
rate of convergence to equilibrium be independent of $\eps$; otherwise the time
needed to reach equilibrium could deteriorate as the medium becomes increasingly
oscillatory. In this paper we answer both questions affirmatively
for a model that incorporates genuine heterogeneities in both the potential function and  the underlying spatial medium.   Concisely, we can establish
(i) the convergence of the $\eps$-dependent gradient flow to an effective
Wasserstein gradient flow, and (ii) the exponential convergence of the
$\eps$-system to its equilibrium with a rate that is uniform in
$\eps$. This is beyond classical homogenization theory as our equation is
both nonlocal and nonlinear. The setting of Wasserstein metric is used as it is particularly tailored to handle the space of probability measures and has found wide range of applications in many models.

Now we describe our model more precisely.
We study the homogenization and long-time behavior of the
following McKean--Vlasov equation of aggregation--diffusion type
\begin{equation}\label{mainFP}
\pt_t \rho^\eps_t  = \nabla \cdot \bbs{\rhoe  \be \nabla (\log \rhoe + U_\eps+W*\rhoe)}.
\end{equation}
Here the interaction kernel is assumed to be symmetric, $W(x,y)=W(y,x)$. 
The potential $U_\eps$ is a fast-oscillating function modeling background
fluctuations, for instance a rough energy landscape. The oscillatory
coefficient $B_\eps$ is motivated by optimal transport in a spatially
inhomogeneous medium. To describe this, consider a cost function $c_\eps(x,y)$
defined through a {\em least action principle},
\begin{equation}
c_\eps(x,y) = \min \left\{ \int_0^1 L_\eps(\dot{z}_t, z_t) \ud t, 
\quad z: [0,1]\longrightarrow\bR^n,\,\,z_0=x, \,\, z_1=y \right\},
\end{equation}
where the Lagrangian $L_\eps$ is a bilinear form in the velocity variable $v$,
given by a positive definite matrix $B_\eps(x)$
\begin{equation}\label{L.bilin}
L_\eps(v,z) = \la B_\eps(z)v, \, v\ra.
\end{equation}
The cost $c_\eps$ then induces a Riemannian distance on $\bR^n$, 
\begin{equation}\label{eps-c}
c_\eps^2(x,y) = \min\left\{ \int_0^1 \la B_\eps(z_t) \dot{z}_t, \dot{z}_t \ra \ud t, \quad z: [0,1]\longrightarrow\bR^n,\,\, z_0 = x, \,\, z_1=y \right\}.
\end{equation}
Note that $L_\eps$ is oscillatory (periodic) in the state variable $z$,
corresponding to transport in an inhomogeneous medium with a periodic structure.
With the above $c_\eps$, we introduce the following $\eps$-Wasserstein distance (squared)
between $\mu,\nu\in \sP(\bR^n)$:
\begin{equation}\label{Wc}
W_\eps^2(\mu,\nu):= \inf\left\{
\iint  c_\eps^2(x,y) \ud \gamma(x,y); \quad \int \gamma(x,\ud y)  = \mu(x), \,\, \int \gamma(\ud x,y) = \nu(y)
\right\}.
\end{equation}
In this current work, we will express \eqref{mainFP} as a gradient flow with respect to $W_\eps$.
Note that when $B_\eps=I$, then $c_\eps(x,y)=|x-y|$ and $W_\eps$ reduces to the classical Wasserstein
distance used in the optimal transport of probability measures. In this case, the corresponding equations
\eqref{eps-c} and \eqref{Wc} were 
first formulated by Benamou and Brenier in \cite{benamou2000computational}. This is essentially
the dynamical point of view of
optimal transport. We find it extremely useful and ``convenient'' if we want to incorporate
spatial inhomogeneity into the model --- all the information is captured
and described by $\be$ or quantities alike.
 
Equation \eqref{mainFP} admits two complementary interpretations. The first is
as a gradient flow on the Wasserstein space $(\sP(\bR^n), W_\eps)$; this
geometric structure will serve as our underlying framework for our analysis of
the $\eps$-Fokker--Planck equation. The second is as the mean-field limit of an
interacting particle process. We describe both in turn.

Consider the free energy $E_\eps$ on the probability space $\sP(\bR^n)$
associated with the McKean--Vlasov equation:
 \begin{equation}\label{E_mc_eps}
 \begin{aligned}
 E_\eps(\rho):=& \int \big[\rho(x)\log  \rho(x) +  U_\eps(x)\rho(x) \big]\ud x+
 \iint \frac12 W(x,y) \rho(y)\rho(x) \ud y\ud x\\
 = & \int \rho(x) \log \frac{\rho(x)}{\pi_\eps(x)} \ud x + \iint \frac12W(x,y)\rho(y)\rho(x) \ud x \ud y - \log \Lambda_\eps,
  \end{aligned}
 \end{equation}
 where  
\begin{equation}\label{pie0}
\pi_\eps(x) := \frac{e^{-U_\eps(x)}}{\Lambda_\eps}, \quad \Lambda_\eps:=\int e^{-U_\eps(x)} \ud x,
\end{equation} 
denotes the invariant measure of the ``single-particle'' dynamics. We remark that
the additive constant
$-\log \Lambda_\eps$ plays no role in our result concerning homogenization, while it will be utilized in the analysis of long time behavior.
We will assume that
$\pi_\eps$ satisfies a Logarithmic Sobolev Inequality --- see \eqref{LSI} ---
which provides a compactness property in the whole space and is the key
ingredient in the proof of long-time convergence.

The first variation
 $\displaystyle \frac{\delta E_\eps}{\delta \rho}$ of $E_\eps$ is given by
\begin{equation}\label{dE}
\frac{\delta E_\eps}{\delta \rho}(\rho) = \log \rho + 1 + U_\eps + \int W(x,y)\rho(y) \ud y = \log \frac{\rho}{\pi_\eps} +1 + W*\rho, 
\end{equation}
where $(W*\rho)(x) = \int W(x,y)\rho(y) \ud y.$
With the same positive definite matrix $B_\eps$ as in \eqref{L.bilin}, 
we consider the following inhomogeneous Fokker--Planck equation,
\begin{equation}\label{epsFP}
\pt_t \rho^\eps_t = \nabla \cdot \bbs{ \rho^\eps_t B_\eps^{-1} \nabla \frac{\delta E_\eps}{\delta \rho}(\rho^\eps_t) } = \nabla \cdot \bbs{\rhoe  \be \nabla (\log \rhoe + U_\eps+W*\rhoe)},
\end{equation}
with initial data $\rho^\eps_0$.

Following Otto's formal Riemannian calculus on Wasserstein space 
\cite{otto2001geometry}, one can identify $\nabla^{W_\eps} E (\rho)$ as
\begin{equation}\label{nablaW}
\nabla^{W_\eps} E_\eps (\rho) := -\nabla \cdot \bbs{\rho B_\eps^{-1} \nabla\frac{\delta E_\eps}{\delta \rho}}.
\end{equation}
In this sense, the inhomogeneous Fokker--Planck equation \eqref{epsFP} can be 
identified as the gradient flow of $E_\eps$ with respect to the $\eps$-Wasserstein metric $W_\eps$, i.e.
\begin{equation}\label{epsGF}
\pt_t \rho^\eps_t = -\nabla^{W_\eps} E_\eps (\rho^\eps_t)
= \nabla \cdot \bbs{\rhoe B_\eps^{-1} (\nabla\log\frac{\rhoe}{\pi_\eps}+\nabla W*\rhoe)},
\end{equation}
where $\nabla W* \rho = \nabla (W*\rho) = \int_y \nabla_x W(x,y) \rho(y) \ud y.$
This is the main underlying framework of the paper.

Note that \eqref{epsGF} can also be interpreted as a
\emph{continuity equation} for the flow of probability mass $\rho^\eps_t$,
\begin{equation}\label{cont.eqn}
\pt_t \rho^\eps_t
+ \nabla \cdot \bbs{\rhoe w^\eps_t}=0,\,\,\,
\text{where $w^\eps_t = -B_\eps^{-1}\nabla \left(\log\frac{\rhoe}{\pi_\eps}+W*\rhoe\right)$.}
\end{equation}
For later purposes, we find it convenient to also write \eqref{epsGF} as
\begin{equation}\label{epsGFu}
\pt_t \rho^\eps_t
= -\nabla \cdot \bbs{\rhoe B_\eps^{-1} \nabla u^\eps_t},\,\,\,
\text{where $u^\eps_t := -\frac{\delta E_\eps(\rhoe)}{\delta \rho} = -(\log\frac{\rhoe}{\pi_\eps}+W*\rhoe)$,}
\end{equation}
which is more in line with the gradient flow interpretation \eqref{nablaW}. 

An immediate consequence of the gradient flow dynamics \eqref{epsGF} is the following \emph{energy dissipation equality}:
\begin{equation}\label{EDE}
\frac{d}{dt}E_\eps(\rho^\eps_t) 
= \int \frac{\delta E_\eps(\rhoe)}{\delta \rho}\partial_t \rho^\eps_t\ud x  =: -D_\eps(\rhoe),
\end{equation}
where explicitly,
\begin{equation}\label{D:Def0}
D_\eps(\rho) = \int
\left\langle\rho B_\eps^{-1}\Big(\nabla\log\frac{\rho}{\pi_\eps}+\nabla W*\rho\Big),
\Big(\nabla\log\frac{\rho}{\pi_\eps}+\nabla W*\rho\Big)
\right\rangle \ud x.
\end{equation}
The key to proving convergence to the stationary solution is the following Logarithmic Sobolev Inequality (LSI):
\begin{equation}\label{LSI0}
E_\eps(\rho) \lesssim D_\eps(\rho).
\end{equation}
Combined with \eqref{EDE}, this immediately implies that the energy   converges
to its minimum value \emph{exponentially fast}, with a rate that is uniform in $\eps$.

A second interpretation of \eqref{epsFP} is as the mean-field limit of a
McKean--Vlasov type interacting particle process with multiplicative noise.
Although in this paper we do not pursue the rigorous mean-field limit as the particle
number tends to infinity, we will describe the underlying stochastic particle system
in order to fix ideas.
Specifically, consider a drift-diffusion-interaction process for $N$ particles with McKean--Vlasov type weak interaction,
\begin{align}
\ud X^{i,N}_t = &b^{i,N}(X^N_t) \ud t + \sigma(X^{i,N}_t) * \ud \mathcal{B}^i_t,
\end{align}
with
\begin{equation}\label{drift.form.ave}
b^{i,N}(x^N) = - \be(x^{i,N}) \bbs{\nabla U_\eps(x^{i,N}) + \frac{1}{N}\sum_j \nabla_x W(x^{i,N},  x^{j,N})}, \quad \sigma(x) = \sqrt{2 \be(x)}.
\end{equation}
Here $\mathcal{B}^i_t$ is an $n$-dimensional Brownian motion
and the multiplicative noise $\sigma(X_t) * \ud \mathcal{B}_t$ is understood in the backward Ito sense,
\begin{equation}
\int_0^t \sigma(X_s)* \ud \mathcal{B}_s = \int_0^t \frac12\nabla \cdot (\sigma \sigma^T)(X_s) \ud s 
+ \int_0^t \sigma(X_s) \ud \mathcal{B}_s\,\,\,\text{for all $t>0$}.
\end{equation}
Formally, letting $N\to\infty$ in the particle system above leads to \eqref{epsFP}. The key idea is the \emph{propagation of chaos} by which individual particles 
asymptotically becomes independent of each other. This is due to fact that in the limit, by
Law of Large Number, the ensemble
average part in the drift \eqref{drift.form.ave} becomes a deterministic function
which can be described by the probability density of a single particle $\rho_t$. Hence each particle sees and interacts with the same function but at the same time is excited by independent (Brownian) noise. We refer to 
\cite{sznitman1991topics, lacker2023hierarchies} for more in depth
discussion of this concept.

Beyond McKean's classical work \cite{McKean1966, McKean1967}, we would point to the results of 
\cite{malrieu2001logarithmic,malrieu2003convergence,liu2021long,guillin2022uniform} which analyze the convergence to mean field equations and also the long time behavior of the solution. Most of these work require some form of convexity of the potential function $U_\eps$
and the interaction kernel $W_\eps$ needs to satisfy conditions to ensure the 
existence of a unique invariant 
measure. If these conditions do not hold, the McKean--Vlasov PDE \eqref{epsFP} will in general undergo phase transitions due to the presence of multiple invariance measures \cite{chayes2010mckean,tugaut2014phase}. 
The analysis will then be much more delicate
as revealed in \cite{CarrilloGvalaniPavliotisSchlichting2020, DelgadinoGvalaniPavliotis2021, DelgadinoGvalaniPavliotisSmith2023}.

In any event, the limiting and long time behavior of systems with multiple scales is an interesting and non-trivial question. This is particularly the case for interacting particle systems as we have further the interaction between the number of particles and the fluctuating or oscillatory length scales. Different limiting procedures can lead to different
descriptions, in particular, for finite and infinite times. 
Among others, see for example \cite{GomesPavliotis2018}, which
considers a closely related setting with rapidly oscillating two-scale confining potential $U_\eps$ and proves that the mean-field and homogenization limits commute for finite times but not in the long-time limit.
We also refer to \cite{DengGaoWang} for distinct homogenization limits of different minimizing movement schemes --- with an additional time discretization parameter --- even for a linear Fokker-Planck equation.

\subsection{Results and contributions}\label{sec:results}
Our main results concern two aspects of the $\eps$-Wasserstein gradient flow
\eqref{epsGF}: its evolutionary Gamma convergence as $\eps\to 0$, and its long-time
behavior for each fixed $\eps$. First, as $\eps \to 0$, the evolution
\eqref{epsGF} converges to a limiting dynamics, which is again characterized as
a gradient flow of an effective total energy $\overline{E}$ with respect to 
an effective Wasserstein distance $\overline{W}$. The distance $\overline{W}$
induced by this evolutionary convergence is still a Riemannian metric on
$\sP(\bR^n)$. Second, for each fixed $\eps$, the solution $\rhoe$ converges to
the stationary solution exponentially fast in time, and both the rate of
convergence and the constants involved are independent of $\eps$.
 
The main approach we use is to first recast \eqref{epsGF} as a  generalized gradient 
flow  in the following form of an \emph{energy dissipation inequality} (EDI)
\begin{equation}\label{EDI1st} 
E_\eps(\rhoe) + \int_0^t \left[ \psi_\eps(\rho^\eps_\tau, \pt_\tau \rho^\eps_\tau ) +  \psi^*_\eps\left(\rho_\tau^\eps, -\frac{\delta E_\eps}{\delta \rho}(\rho_\tau^\eps)\right) \right] \ud \tau  \leq  E_\eps(\rho_0^\eps).
\end{equation}
This formulation involves dissipation functionals $\psi_\eps$ and 
$\psi^*_\eps$ on the tangent and the co-tangent plane of $\sP(\bR^n)$, 
respectively. 
Inequality \eqref{EDI1st} is in fact equivalent to the strong form of gradient flow \eqref{epsGF} since the functional $\psi_\eps$ and $\psi^*_\eps$ are convex conjugate of each other.
Then the limiting description of the dynamics is obtained by passing to the limit in the functionals appearing in
\eqref{EDI1st}.
As of long time behavior, as mentioned earlier, the key is the LSI \eqref{LSI0}.
We do remark that we are working in a perturbative regime in the sense that $W$
is small so that the overall energy landscape is
convex and hence there is a unique stationary solution.

We now give some comments about our result in comparison with other existing works.
In our previous paper \cite{GaoYip}, we considered a similar homogenization 
problem, but for a single-particle case in which the Fokker-Planck equation becomes a
linear equation -- \eqref{mainFP} without the $W*\rhoe$ term. In the same vein, it is interpreted as a Wasserstein gradient flow.
We use the same EDI framework to characterize the limiting equation as a gradient flow
with respect to an effective Wasserstein distance $\overline{W}$. In particular, we found
out that in general there will be discrepancy between $\overline{W}$ and the \emph{Gromov--Hausdorﬀ limit $W_{\text{GH}}$} of the $\eps$-Wasserstein distance $W_\eps$. It would be intersting to understand such a difference in a more quantitative manner.
Another comment is that the $\eps$-equation involves both \emph{energetic} and 
\emph{kinematic effects}. The former is tied to $U_\eps$ while the latter to $\nabla W^\eps$ which is
connected to $B_\eps$. In general, these two effects will interact with each other, in particular when the underlying invariant measure $\pie$ converges only weakly. On the
other hand, if $\pie$ converges strongly, such interaction disappears. This can also be seen in the formula for the effective diffusivity coefficient -- see 
\eqref{eff.rho.eqn} and \eqref{eff.rho.eqn0} in the Appendix.

Homogenization or finding effective dynamics certainly has a long history. 
We refer to \cite{bensoussan2011asymptotic} for a classical treatise on this topic. In essence, the key
is to find the \emph{corrector} --- solution to an appropriate cell problem. Then the
effective coefficient can be expressed as some average over the cell problem. On the other hand, for nonlinear problems, finding the cell problem, its solution and prove the convergence can be highly nontrivial as both spatial and temporal scales interact with
each other. For problems with variational structure, the
framework of $\Gamma$-convergence is extremely robust and versatile.
As mentioned before, this idea has been
extended to evolutionary equations, see for example, \cite{braides2014local,mielke2016evolutionary} for
some overviews of this approach.
In terms of homogenization of Wasserstein distance, from our perspective, the most relevant work are \cite{weinan1991class} and \cite{gangbo2012homogenization}. Our work seems to be
the first to show the homogenized limit for gradient flow problem in Wasserstein space.

A common technique to analyze long time behavior is to make use of functional inequality such as 
\eqref{LSI0}.
This has been applied to Fokker-Planck and McKean-Vlasov type equation in \cite{CarrilloMcCannVillani2003}. {\blue Their} work crucially rely on the convexity of the potential function so that the underlying energy is \emph{geodesically convex}. See also 
\cite{villani2003topics, villani2009optimal} for general monographs on this and related 
idea. In two much earlier work \cite{Tamura84,Tamura87}, Tamura had proved the
exponential convergence to invariant measure for a general McKean-Vlasov equation
in $L^2$ and $L^1$-norms. 
The potential function $U\sim|x|^\alpha$ has sufficient growth at infinity so that
the author can make use of \emph{hyper-} and 
\emph{ultra-contractivity} of the linearized operator at the invariant measure: $\alpha=2$ for the former while $\alpha>2$ for the latter case. These two works of Tamura correspond to our
$\eps=1$.

The convexity requirement for the energy landscape (in Wasserstein space) seems to break down under the influence of oscillatory media. Fortunately, under bounded perturbation of the potential function, LSI still holds. Based on this, we are able to modify the approach 
of \cite{Tamura87} to prove the convergence to an invariant measure exponentially fast. 
The initial condition is fairly general (with finite entropy and second moment). The proof is in Section \ref{longtime} where we will also give more detail description of Tamura's
results.

Before leaving this introduction section, we provide a roadmap of our paper. 
Section \ref{sec2} gives detail formulation of our equation and
description of results in terms of an energy dissipation inequality.
We also introduce the necessary assumptions and notations.
Section \ref{sec3} gives some preliminary results, apriori estimates and
basic convergence of our solution. 
We would in particular highlight here the properties of invariant measures regarding to Logarithmic Sobolev Inequality.
Section \ref{epsEDI-0EDI} provides the proof of our main convergence result
Theorem \ref{thm1} showing the validity of a limiting Wasserstein gradient flow.
Section \ref{longtime} contains the proof of our long time convergence
result Theorem \ref{thm2}. In the Appendix, we give sufficient detail proof for a technical lemma from \cite{Tamura87} and also the computation of
the limiting effective coefficient using formal asymptotic analysis.

\section{
Formulation of generalized gradient flow, Assumptions and Main results}\label{sec2}
In this section, we first formulate the inhomogeneous McKean-Vlasov PDE \eqref{epsGF} as a 
generalized gradient flow. Then we describe our main results on the homogenization and long time behavior of PDE \eqref{epsGF}. We finally clarify   the assumptions and notations of our settings.

\subsection{$\eps$-generalized gradient flow in energy-dissipation inequality (EDI) form}\label{sec2.2}

In this section, we describe more precisely the gradient flow approach used in the current paper.
In particular, we will reformulate \eqref{epsGF} using an equivalent form of energy-dissipation inequality (EDI). 
As motivated earlier, we find such an approach particularly useful when investigating the limiting behavior of \eqref{epsGF} as $\eps\to0$.

We introduce the $\eps$-dissipation on the \emph{tangent plane} $T_{\sP}$ as a functional $\psi_\eps: \sP \times T_{\sP} \to \bR$ defined by
\begin{equation}\label{psi}
\psi_\eps(\rho, s):= \frac12 \int \la \nabla u, B_\eps^{-1} \nabla u \ra \rho \ud x, \quad \text{ with }\, s=-\nabla \cdot \bbs{\rho B_\eps^{-1} \nabla u},
\end{equation}
and the $\eps$-dissipation on the \emph{cotangent plane} $T^*_{\sP}$ as a functional $\psi^*_\eps: \sP \times T^*_{\sP} \to \bR$ defined by
\begin{equation}\label{c_psi}
\psi^*_\eps(\rho, \xi):= \frac12 \int \la \nabla \xi, B_\eps^{-1} \nabla \xi \ra \rho \ud x.
\end{equation}
Note that in \eqref{psi}, given $\rho$ and $s$, the function $u$ might not exist. But then
$\psi$ can be interpreted as the Legendre transform of $\psi^*$  which will
be taken as the definition of $\psi_\eps$:
\begin{equation}\label{psi-def}
\psi_\eps(\rho,s) = \sup_{\xi\in T^*_\rho}\,\, \la \xi, s \ra_{T^*_\rho, T_\rho} - \psi^*_\eps(\rho, \xi).
\end{equation}
If the $\sup$ in \eqref{psi-def} is attained at $\xi^*$, then $s= -\nabla \cdot \bbs{\rho B_\eps^{-1}\nabla \xi^* }$ and 
\begin{equation}
\psi_\eps(\rho,s) 
= \la \xi^*, s \ra_{T^*_\rho, T_\rho} - \psi^*_\eps(\rho, \xi^*)
= \frac12 \int \la \nabla \xi^*, B_\eps^{-1} \nabla \xi^* \ra \rho \ud x.
\end{equation}
On the other hand, if $s$ takes special form as $s=-\nabla \cdot \bbs{\rho B_\eps^{-1} \nabla u}$ for some $u$, then $\psi$ as defined 
in \eqref{psi-def} is indeed given by \eqref{psi}:
\begin{eqnarray*}
\la \xi, s \ra_{T^*_\rho, T_\rho} - \psi^*_\eps(\rho, \xi)
&=&
\int \la \nabla \xi, B_\eps^{-1} \nabla u \ra \rho \ud x
-\frac12\int \la \nabla \xi, B_\eps^{-1} \nabla \xi \ra \rho \ud x\\
&=&
\frac12\int \la \nabla u, B_\eps^{-1} \nabla u \ra \rho \ud x
-\frac12\int \la \nabla (\xi-u), B_\eps^{-1} \nabla (\xi-u) \ra \rho \ud x
\end{eqnarray*}
the $\sup$ of which is attained at $\xi=u$.
Hence \eqref{psi} is consistent with
\eqref{psi-def}.

Applying the Fenchel-Young inequality to the convex functionals 
$\psi_\eps$ and $\psi^*_\eps$, we have
\begin{equation}\label{FenchelYoung}
\la\xi, s\ra \leq \psi_\eps^*(\rho, \xi) + \psi_\eps(\rho, s),
\quad\text{for all}\quad \xi\in T_\rho^*,\,\,\,\text{and}\,\,\,s\in T_\rho,
\end{equation}
with equality holds if and only if
$\xi\in\partial_s\psi_\eps(\rho,s)$ which is equivalent to
$s\in\partial_\xi\psi^*_\eps(\rho,\xi)$.
Here 
$\partial_s\psi_\eps(\rho, s)$ 
and $\partial_\xi\psi_\eps^*(\rho, \xi)$ refer
to the sub-differentials of $\psi_\eps$ and $\psi_\eps^*$ on
$T_\rho$ and $T_\rho^*$, respectively, at a fixed $\rho$. 
More precisely, $\xi\in\partial_s\psi_\eps(\rho, s)$ means
$\xi$ satisfies $s=-\nabla\cdot(\rho B_\eps^{-1}\nabla\xi)$ 
which is also equivalent to $s\in\partial_\xi\psi^*_\eps(\rho, \xi)$.

With the above, we write \eqref{epsGF} in the form of an EDI. Consider
\begin{equation}\label{eng.diss.eqn}
\frac{d}{dt}E_\eps(\rhoe) 
= \left\langle\frac{\delta E_\eps}{\delta\rho},\pt_t\rhoe\right\rangle,
\,\,\,\text{or}\,\,\,
\frac{d}{dt}E_\eps(\rhoe) 
+ \left\langle-\frac{\delta E_\eps}{\delta\rho},\pt_t\rhoe\right\rangle = 0.
\end{equation}
(We recall form \eqref{dE} for $\frac{\delta E_\eps}{\delta\rho}$.)
By \eqref{FenchelYoung}, $\pt_t\rhoe = -\nabla^{W_\eps}E_\eps(\rhoe)=\nabla \cdot \bbs{\rho B_\eps^{-1} \nabla\frac{\delta E_\eps}{\delta \rho}}$ 
\emph{if and only if}
\[
\psi_\eps(\rho^\eps_\tau, \pt_\tau \rho^\eps_\tau ) +  \psi^*_\eps\left(\rho_\tau^\eps, -\frac{\delta E_\eps}{\delta \rho}(\rho_\tau^\eps)\right)
\leq \left\langle-\frac{\delta E_\eps}{\delta\rho},\pt_t\rhoe\right\rangle.
\]
Hence, upon integrating \eqref{eng.diss.eqn}, our gradient flow \eqref{epsGF} 
is equivalent to the following:
\begin{equation}\label{epsEDI}
E_\eps(\rhoe) + \int_0^t \left[ \psi_\eps(\rho^\eps_\tau, \pt_\tau \rho^\eps_\tau ) +  \psi^*_\eps\left(\rho_\tau^\eps, -\frac{\delta E_\eps}{\delta \rho}(\rho_\tau^\eps)\right) \right] \ud \tau  \leq  E_\eps(\rho_0^\eps). 
\end{equation}

Before leaving this section, we write down the following explicit expressions:
\begin{eqnarray}
\psi^*_\eps\left(\rho_\tau^\eps, -\frac{\delta E_\eps}{\delta \rho}(\rho_\tau^\eps)\right) 
&=& \int \left\langle
\nabla\left(\frac{\delta E_\eps}{\delta\rho}\right),
B_\eps^{-1}
\nabla\left(\frac{\delta E_\eps}{\delta\rho}\right)
\right\rangle\rhoe
\,\ud x\nonumber\\
&=&
\frac12 \int \left\la \nabla\left(\log \frac{\rho^\eps_\tau}{\pie} + W\ast\rho^\eps\right), \be 
\nabla\left(\log \frac{\rho^\eps_\tau}{\pie} +W\ast\rho^\eps\right)\right\ra \rho^\eps_\tau\ud x,\label{psi-*-eqn}
\end{eqnarray}
and
\begin{equation}
\psi_\eps\left(\rho^\eps_\tau, \pt_\tau\rho^\eps_\tau\right)
= \frac12 \int \la \nabla u^\eps_\tau, B_\eps^{-1} \nabla u^\eps_\tau \ra \rho_\tau^\eps 
\ud x, \quad \text{ with }\, -\nabla \cdot \bbs{\rho_\tau B_\eps^{-1} \nabla u^\eps_\tau} = \pt_\tau\rho_\tau^\eps.\label{psi-eqn}
\end{equation}

\subsection{Main results}\label{main}
Briefly stated, our main result is that the gradient flow structure is
preserved in the limit, i.e., \eqref{epsGF} converges to a \emph{limiting
gradient flow} as $\eps\to0$. 
We also prove that for both the $\eps$ and limit problems, the solution 
$\rhoe$ and $\rho_t$ converge \emph{exponentially fast} as $t\to\8$ to their respective stationary solutions. The rate of convergence can be characterized in terms of
the constants appearing in some Logarithmic Sobolev Inequality (LSI). Their estimates
are uniform in $\eps$. The results are briefly described as follows.

\begin{enumerate}
\item
(Theorem \ref{thm1}) The solution $\rho^\eps_t$ of 
\eqref{epsGF} converges (weakly) to $\rho_t$ that solves a gradient flow
with respect to a limiting Wasserstein distance $\lbar{W}$,
\begin{equation}\label{0GF}
\pt_t \rho_t 
= -\nabla^{\lbar{W}}\lbar{E}(\rho_t) 
= \nabla \cdot \left[\rho_t \lbar{B}^{-1} \nabla\left(
\log\frac{\rho_t}{\lbar{\pi}} + W*\rho_t
\right)\right]
.
\end{equation}
In the above, the limiting energy is given as
\begin{equation}\label{0Eng}
\lbar{E}(\rho) := 
\int \rho\log\frac{\rho}{\lbar{\pi}}\,\ud x + \frac12\iint  W(x,y)\rho(y)\ud y \rho(x) \ud x,
\end{equation}
where the $\lbar{\pi}$ is simply the spatial average of $\pi_\eps$
with respect to some fast variable -- see \eqref{pi.ave} below.
The matrix $\lbar{B}$ is obtained by taking
appropriate average of $B_\eps$ over the fast variable weighted by the 
solution of a cell problem \eqref{eff.A} or equivalently, by considering 
the $\Gamma$-limit of a variational functional (Theorem \ref{GammaMainThm}).
The Wasserstein distance $\lbar{W}$ is related to $\lbar{B}$ in the same way as
$W_\eps$ to $B_\eps$ -- see \eqref{nablaW}. 
Hence we also have
\begin{equation}
\nabla^{\lbar{W}}\lbar{E}(\rho) =-
\nabla \cdot \left[\rho\lbar{B}^{-1} \nabla\left(
\frac{\delta\lbar{E}}{\delta\rho}(\rho)
\right)\right],\quad\text{where}\quad
\frac{\delta\lbar{E}}{\delta\rho}(\rho)
=\log\frac{\rho}{\lbar{\pi}} + W*\rho.
\end{equation}

Similar to \eqref{epsEDI}, \eqref{psi}, and \eqref{c_psi}, 
equation \eqref{0GF} is equivalent to an EDI, i.e.,
\begin{equation}\label{0EDI}
\lbar{E}(\rho_t) + \int_0^t \left[ \psi(\rho_\tau, \pt_\tau \rho_\tau ) 
+  \psi^*\left(\rho_\tau, -\frac{\delta \lbar{E}}{\delta \rho}(\rho_\tau)\right) \right] 
\ud \tau  \leq  \lbar{E}(\rho_0),
\end{equation}
where $\psi^*: \sP \times T^*_{\sP} \to \bR$ is the limiting dissipation 
functional on the cotangent plane $T^*_{\sP}$ similarly defined by
\begin{equation}\label{c_psiL0}
\psi^*(\rho, \xi):= \frac12 \int \la \nabla \xi, \bar{B}^{-1} \nabla \xi \ra \rho \ud x,
\end{equation}
and $\psi: \sP \times T_{\sP} \to \bR$ is the limiting dissipation 
functional on the tangent plane $T_{\sP}$, given similar to \eqref{psi-def} by
\begin{equation}\label{psiL0}
\psi(\rho, s):= 
\sup_\xi \left\{\int \xi s\ud x-
\frac12 \int \la \nabla \xi, \bar{B}^{-1} \nabla \xi \ra \rho \ud x\right\}.
\end{equation}
By \eqref{0EDI}, for almost every $\tau$, equality holds in the Fenchel-Young inequality for $\psi$ and $\psi^*$:
\[
\left\la -\frac{\delta \lbar{E}}{\delta \rho}(\rho_\tau), \pt_\tau \rho_\tau\right\ra
=
\psi(\rho_\tau, \pt_\tau \rho_\tau ) 
+  \psi^*\left(\rho_\tau, -\frac{\delta \lbar{E}}{\delta \rho}(\rho_\tau)\right).
\]
Hence, 
$s\in\partial_\xi\psi^*\left(\rho_\tau, -\frac{\delta \lbar{E}}{\delta \rho}(\rho_\tau)\right)$, i.e. $s=-\nabla\left(\rho\bar{B}^{-1}\nabla(-\frac{\delta \lbar{E}}{\delta \rho})\right)$.

Inequality \eqref{0EDI} is derived as the limit of \eqref{epsEDI}.
The proof relies on $\Gamma$-convergence type argument which is well-known for stationary problems. Both variational lower- and upper-bounds are utilized.
Then a technical lemma is invoked to extend the result to time dependent case. We remark that our result is also consistent with the statement obtained by 
using asymptotic expansion described in Appendix \ref{asym.exp}.

\item
(Theorem \ref{thm2})
The energy and solution of \eqref{epsGF} converge to their limiting values exponentially fast:
\[
E_\eps(\rhoe)-E_\eps(\mu_\eps) \lesssim e^{-\frac{2\lambda}{\alpha_B}t},
\quad\text{and}\quad
    \|\rhoe- \mu_\eps\|^2_{L^1}   \lesssim e^{-\frac{2\lambda}{\alpha_B}t},
\]
where $\mu_\eps$ is the invariant measure, or equivalently the stationary
solution of \eqref{mainFP}, $\lambda$ is the LSI constant for $\pie$ and 
$\alpha_B$ is related to estimates for $B_\eps$. Both of the above hold uniformly in the limit
$\eps\to0$.

A brief remark about the proof is in place. The key is the
energy dissipation functional \eqref{EDE} and the LSI for $\pie$ \eqref{LSI0}. Algebraically, the energies in these two expressions are not exactly the same --- see \eqref{Erelation}. We resort to the technical result \cite[Lemma 6.7]{Tamura87} to handle this discrepancy. In the hindsight, we rely on the non-degeneracy or convexity of the energy at single particle invariant measure $\pie$.
We then have the exponential decay of energy to its minimum. An application of Pinsker's inequality gives the
$L^1$-convergence. 

\end{enumerate}

\subsection{Assumptions and Notations.}\label{sec:assumptions}
We introduce here the conventions and notations used in this paper. 

First, the symbol $\sP(\mathbb R^n)$, or simply $\sP$, refers to the space of probability measures on $\mathbb R^n$. Since we will be able to establish sufficient regularity for our
solution $\rhoe$ (and its limit $\rho_t$), without loss of generality, any element 
$\rho\in\sP(\mathbb R^n)$ will henceforth be assumed to possess a density with respect to the Lebesgue measure. 
In order for the Wasserstein distance $W_\eps$ \eqref{Wc} to be well defined, we need to consider probability measures with finite second order moment, i.e.,
$\displaystyle \int |x|^2\rho(x)\ud x < \8.$ Elements in $\sP(\mathbb R^n)$ satisfying this
condition will be denoted by $\sP_2(\mathbb R^n)$. Note that if $B_\eps=I$, then 
$W_\eps$ corresponds to the classical $W_2$ for which the cost function $c_\eps$
is simply given by $|x-y|$:
\[
|x-y|^2= \min\left\{ \int_0^1 \la \dot{z}_t, \dot{z}_t \ra \ud t, \quad z: [0,1]\longrightarrow\bR^n,\,\, z_0 = x, \,\, z_1=y \right\}.
\]
We would sometimes use $(\sP_2, W_\eps)$ or $(\sP_2, W_2)$ if we would like to emphasize 
the underlying metric. Otherwise, $\sP$ is endowed with the usual narrow (or weak) topology.

Second, common function spaces used in this paper are $L^1(\mathbb R^n)$, $L^2(\mathbb R^n)$, $L^\infty(\mathbb R^n)$ and relevant Sobolev spaces.
We will also consider the following weighted spaces $L^2$-spaces,
\[
\|f\|_{L^2(\sigma)}=\left(\int |f(x)|^2\sigma(x)\ud x\right)^{\frac{1}{2}},\quad
\|f\|_{H^1(\sigma)}
=\left(\int |\nabla f(x)|^2\sigma(x)\ud x\right)^{\frac{1}{2}}+\|f\|_{L^2(\sigma)}
.
\]
The weight $\sigma$ will be taken to be some underlying invariant measures. 
Unless otherwise stated, the domain of integration will always be $\mathbb R^n$. For simplicity, $L^p$ without any weight means $L^p(\mathbb R^n)$.

Third, as we will consider functions that oscillate on a 
small length scale, $0 < \eps \ll 1$, it is convenient to introduce the 
following fast variable
\begin{equation}\label{fast.var}
y:=\frac{x}{\eps}.
\end{equation}
The domain for $y$ is taken to be the $n$-dimensional torus $\mathbb T^n$
when the oscillatory functions are $1$-periodic in $y$.
The symbols $\wra$ and $\lra$ refer to weak and strong 
convergence in appropriate function spaces. 
For the convergence of a sequence of functions $f_\eps$ as $\eps\to0$, 
we will use the same notation even if the convergence only holds upon
extraction of subsequence. The convergence can be established for the whole
sequence if the limiting equation has unique solution which is the case
for our McKean-Vlasov equation \eqref{0GF}.

Next we state the general assumptions for our settings.
They are used mostly for concreteness and can be much
relaxed if we opt to use more technical tools.
More precise and quantitative requirements will be given in respective results.

\begin{enumerate}[(i)]
\item
The matrix $B_\eps$ is assumed to take the form,
\begin{equation}
\label{Bform}
B_\eps(x) = B\left(\frac{x}{\eps}\right),\,\,\,\text{or}\,\,\,
B_\eps(x) = B(y),
\end{equation}
where $B(\cdot)$ is $1$-periodic.
Furthermore, $B(\cdot)$ is bounded and uniformly positive definite, i.e.,
there is a constant $\alpha_B>0$ such that for all $y\in\mathbb T^n$, it holds that
\begin{equation}\label{B.est}
\frac{1}{\alpha_B}I \leq B(y) \leq \alpha_B I.
\end{equation}
The constant $\alpha_B$ can be chosen such that
$\|B\|,\,\|B^{-1}\| \leq \alpha_B$.

\item
The potential function $U_\eps$ is bounded from below and has sufficient growth at $\8$ 
to ensure
$\int e^{-U_\eps}\ud x < \8$ so that the (single-particle) Gibbs measure $\pie$ 
\eqref{pie0} is well-defined. 
Though it can certainly be more general, for concreteness, 
we consider $U_\eps$ taking the following additive form:
\begin{equation}\label{Ue}
U_\eps(x) = U_0(x) + V_\eps(x),
\end{equation}
where $U_0$ is some fixed function on $\mathbb R^n$ and 
$V_\eps(x) = V(y)\in L^\8(\mathbb T^n)$. Let
\begin{equation}\label{pi0}
\pi_0(x):=\frac{1}{\Lambda_0}e^{-U_0(x)},\,\,\,\text{where}\,\,\,
\Lambda_0:=\int e^{-U_0(x)}\ud x < \infty.
\end{equation}
We further introduce $\beta_V:=\|V\|_{L^\8(\mathbb T^n)}$.

With the above, the single particle invariant measure $\pie$ \eqref{pie0} can be written as:
\begin{equation}
\label{pie} 
\pie(x) 
=\frac{1}{\Lambda_\eps}e^{-U_0(x)-V(\frac{x}{\eps})}
=\frac{1}{\Lambda_\eps}e^{-U_0(x)}e^{-V(\frac{x}{\eps})}
=:\pi\left(x,\frac{x}{\eps}\right)
\end{equation}
Then the weak limit of $\pie$ is given by averaging over the fast
variable:
\begin{equation}\label{pi.ave}
\lbar{\pi}(x) := \text{wk-}\lim_\eps\pie(x)=\int \pi(x,y)\ud y
\end{equation}
which is exactly equal to $\pi_0$.

We further assume that $U_0, V$ (and hence $\pi$), and $B$ are all smooth enough functions in the $x$ and $y$ variables.

Two further remarks are in place. First, note that in \eqref{Ue}, $V(y)$ is some $L^\8$-function defined on $\mathbb T^n$. Hence in general $U_\eps$ and $\pie$ do not
converge strongly. On the other hand, if we make the more restrictive assumption that $V_\eps(x) = \eps V(\frac{x}{\eps})= O(\eps)$, then $U_\eps$ and $\pie$ do have strong limits.
We refer respectively the former as Type I (oscillatory) and the latter as Type II (uniform) case. The results in this paper cover the Type I case but we will also comment on the results in the simpler, Type II case -- see the end of 
Appendix \ref{asym.exp}.

Second, the long time behavior of $\rhoe$ relies on functional inequalities related to $\pie$, in particular
$\pie$ satisfies LSI. This requires $U_\eps$ (or $U_0$ in our case) 
to have sufficient growth at $\8$. A specific example is the Gaussian case, 
$U_0,\,U_\eps \sim e^{-|x|^2}$. See the more general discussion after
Lemma \ref{lem:LSI-single}.

\item We consider symmetric interacting kernel $W$, 
$W(x,y)=W(y,x)$.
It is further assumed to be bounded and Lipschitz. We henceforth define the following finite constants,
\begin{equation}\label{W.est}
    \gamma_W:=\|W\|_{L^\8(\mathbb R^n\times\mathbb R^n)},\,\,\,\text{and}\,\,\,
    \delta_W:=\|\nabla W\|_{L^\8(\mathbb R^n\times\mathbb R^n)} < \8.
\end{equation}
It turns out we need $\gamma_W$ and $\delta_W$ to be appropriately small
to ensure uniqueness of the invariant measure $\mu_\eps$ 
and for the long time convergence of $\rhoe$. However, the smallness does not
depend on $\eps$. See Section \ref{longtime} for more precise statement.

\item	The initial data $\rho^\eps_0$ is assumed to be
well-prepared in the following sense:
\begin{equation}\label{id_con}
\text{there is a $\rho_0$ such that 
as $\eps\to0$, it holds that $\rho^\eps_0 \wra \rho_0$
and
$E_\eps(\rho^\eps_0) \to \lbar{E}(\rho_0)<\8.$}
\end{equation}
Here $E_\eps$ and $\lbar{E}$ are given by \eqref{E_mc_eps} and \eqref{0Eng}.

Note that as the entropy function has super-linear growth, $\rho_0^*$ automatically
has a density with respect to the Lebesgue measure. This is also consistent with the 
Dunford–Pettis Theorem -- see also \cite[p.151]{ABS}. As mentioned at the beginning of this section, due to regulatory property of our solution, such a 
property will be preserved for $t>0$.
\end{enumerate}

We will also use the following general conventions.
\begin{itemize}
\item For time dependent problems, we are dealing with
functions depending on both the space and time variables $x,t$. For ease of notation, 
given a function $f=f(x,t)$, we often use $f_t$ to denote $f(t,\cdot)$, i.e., 
the slice of $f$ at a fixed time $t$.

\item The symbol $C$ refers to some generic constant that could depend on 
$\alpha_B$, $\beta_V$, $\gamma_W$, $\delta_W$ and initial data. 
The dependence on specific variable or quantity will be made explicit if that can reveal extra information. For example, dependence on the time variable $t$ will be denoted by $C_t$ and by $C_T$ for $0< t < T$.

\item The notation $\lesssim$ means that the inequality holds up to some
universal multiplicative constant. The value of the constant can change from  
line to line.
\end{itemize}

\section{Some a-priori estimates}\label{sec3}
In this section, we analyze properties of invariant measures $\mu_\eps$ and provide
some a prior estimates of the solution $\rhoe$. All the essential constants will be
shown to be independent of $\eps$.

For the rest of this paper, we will make use of the assumptions (i)--(iv) introduced in the previous section. In particular, we recall here the constants $\alpha_B$ \eqref{B.est},
$\beta_V$, and $\gamma_W$ and $\delta_W$ \eqref{W.est} defined for the quantities $B$, $V$, and $W$.

\subsection{Existence and properties of invariant measures $\mu_\eps$}

Recall the free energy $E_\eps$ \eqref{E_mc_eps} and its first variation $\delta E_\eps$ \eqref{dE}. Any minimizer $\mu_\eps$ of $E_\eps$ satisfies the following 
equation
\begin{equation}\label{statMF}
    \mu_\eps = \frac{1}{\Xi_\eps} e^{-U_\eps - W*\mu_\eps}, \quad \text{ with } \, \Xi_\eps=\int e^{-U_\eps - W*\mu_\eps} \ud x. 
\end{equation}
Such a $\mu_\eps$ is also a stationary solution for the
McKean-Vlasov equation \eqref{epsFP}. 
As the above is a nonlinear equation, it is not apriori clear if a solution exists and is unique. To understand $\mu_\eps$, we find it instructive
to first analyze the invariant measure $\pi_\eps$ \eqref{pie0} for the single particle case. 
 
The first result presents some simple but useful properties of the energy functional
$E_\eps$.

\begin{lem}\label{lem:E}
Let the free energy $E_\eps$ be defined as \eqref{E_mc_eps}. 
It has the following properties:
\begin{enumerate}[(i)]
    \item $E_\eps$ is   lower semicontinuous with respect to narrow convergence in $\sP(\bR^n)$ (the duality with $C_b(\bR^n)$);
    
    \item $\{\rho; \,E_\eps(\rho)\leq C\}$ is sequentially compact with respect to narrowly convergence in $\sP(\bR^n)$ (the duality with $C_b(\bR^n)$);
    
    \item If in addition, $\gamma_W<1$, then $E_\eps$ is strictly convex.
\end{enumerate}

\end{lem}
\begin{proof}
    Statements (i) and (ii) are standard -- they are in fact stated in \cite[Theorem 3.1]{Tamura87}. We refer to \cite[Chapter 15]{ABS} for a general exposition.

    For statement (iii), we compute the second variation of $E_\eps$
    \begin{align*}
        \frac{\ud^2}{\ud s^2} E_\eps(\rho+s q)=& \int \frac{q^2}{\rho} \ud x + \int (W*q)q \ud x \\
        \geq & (\int |q| \ud x)^2 + \int (W*q)q \ud x,
    \end{align*}
    where we used    Jensen's inequality for $\int (\frac{q}{\rho})^2 \rho \ud x$.
    Since $\|W*q\|_{L^\8}\leq \|W\|_{L^\8}\|q\|_{L^1}$, we further  have
    \begin{align*}
        \frac{\ud^2}{\ud s^2} E_\eps(\rho+s q) 
        \geq    \bbs{1-  \|W\|_{L^\8}}\bbs{\int |q| \ud x}^2,
    \end{align*}
    which by $\|W\|_{L^\8}<1$, concludes that
    $E_\eps$ is strictly convex.
    \end{proof}
    
By the direct method in calculus of variations and Lemma \ref{lem:E}, we conclude 
\begin{cor}\label{cor:E}
    If $\gamma_W<1$, then there exists a unique minimizer, denoted as $\mu_\eps$, of the free energy $E_\eps$. 
\end{cor}
It is an interesting question to consider the case when the invariant measure is
non-unique. We refer to Section \ref{sec:results} for some general discussion about this.
 
Next, we consider Log-Sobolev Inequality (LSI) associated with an underlying
probability measure $\pi$. The overall theory is initiated by Bakry-Emery theory \cite{BakryEmery1985}. Precisely, we say $\pi\in \sP(\bR^n)$ satisfies LSI($\lambda$) if
\begin{equation}\label{LSI}
    \int \rho \log \frac\rho{\pi} \ud x \leq \frac{1}{2\lambda}\int \rho |\nabla \log \frac\rho{\pi}|^2\ud x, \quad \forall \rho \in \sP(\bR^n).
\end{equation}
Another useful version of the above is given in terms of the function
$f:=\frac{\rho}{\pi}$:
\begin{equation}
\int f (\log f) \pi\ud x \leq \frac{1}{2\lambda}\int \frac{|\nabla f|^2}{f}\pi\ud x
\,\,\left(=\frac{2}{\lambda}\int |\nabla\sqrt{f}|^2\pi\ud x\right).
\end{equation}
The right hand side of the above is often called the \emph{Fisher Information}.
From \eqref{LSI}, a simple calculation using the representation $\frac{\rho}{\pi} = 1 + \eps g$ with $\int g\pi \ud x=0$ and $\eps\ll 1$, implies the following Poincare inequality,
\begin{equation}\label{PI}
    \|g\|_{L^2(\pi)}^2 \leq \frac{1}{\lambda}\|\nabla g\|_{L^2(\pi)}^2.
\end{equation}
We refer to \cite[Chapter 9]{villani2003topics} for a general introduction to the 
theory of LSI, its connection to optimal transport and applications to nonlinear PDEs.
 
As a starting point, recall the form of the potential function 
$U_\eps = U_0 + V_\eps$ \eqref{Ue} and the Gibbs measures $\pi_0$ and $\pie$
\eqref{pi0} and \eqref{pie}.
Suppose $\pi_0$ satisfies LSI. We would like to extend LSI from $\pi_0$ to $\pie$.
This is achieved in the next result. 
It is noted in many places, 
see for instance \cite[Theorem 9.9(ii)]{villani2003topics} and \cite[Proposition 5.1.6]{bakry2014analysis},
but for reader's convenience, we provide a proof.

\begin{lem}[LSI for $\pi_\eps$]\label{lem:LSI-single}
Suppose $\pi_0$ satisfies LSI($\lambda_0$) for some constant
$\lambda_0$, then $\pie$ (defined in \eqref{pie}) 
satisfies LSI($\lambda$) with $\lambda = \lambda_0 e^{\min|\ve|-\max|\ve|}$.
\end{lem}
\begin{proof}
   First, denote $\pi_0=\frac{e^{-U_0}}{\int e^{-U_0}}$. By assumption, we have
    \begin{equation}\label{LSI_pi0}
    \int \rho \log \frac\rho{\pi_0} \ud x \leq \frac{1}{2\lambda_0}\int \rho |\nabla \log \frac\rho{\pi_0}|^2\ud x, \quad \forall \rho \in \sP(\bR^n).
\end{equation}
Denote $b:=\int \pio e^{-V_\eps} \ud x.$ Then $\pie = \frac{\pio e^{-V_\eps}}{b} \in \sP(\bR^n)$.
By \eqref{LSI_pi0}, for  $\frac{b \rho e^{V_\eps}}{a}\in \sP(\bR^n)$ with $a:=b \int \rho e^V \ud x$, we know
\begin{align}\label{tm_lhs}
    \text{LHS}:=\int \frac{b \rho e^{V_\eps}}{a} \log \frac{b \rho e^{V_\eps}}{a \pio} \ud x \leq \frac{1}{2\lambda_0} \int \frac{b \rho e^{V_\eps}}{a} \left|\nabla \log \frac{b \rho e^{V_\eps}}{a \pio} \right|^2 \ud x\leq \frac{e^{\max |V_\eps|}}{2\lambda_0} \frac{b}{a} \int \rho |\nabla \log \frac{\rho}{\pie}|^2 \ud x.
\end{align}

Second, recast the relative entropy as
\begin{align}\label{entL}
  \int \rho \log \frac{\rho}{\pie} =\int \rho \log \bbs{ a \frac{b\rho e^{V_\eps}}{a\pio} } \ud x = \log a + \int \rho \log \frac{b\rho e^{V_\eps}}{a\pio}\ud x.  
\end{align}
We estimate the last term above through
\begin{align*}
   b\int \rho \log \frac{b\rho e^{V_\eps}}{a\pio}\ud x =& \int \frac{b\rho}{a\pio e^{-V_\eps}} \log \frac{b\rho e^{V_\eps}}{a\pio} a\pio e^{-\ve} \ud x\\
   =& \int \left(\frac{b\rho}{a\pio e^{-V_\eps}} \log \frac{b\rho e^{V_\eps}}{a\pio} - \frac{b\rho}{a\pio e^{-\ve}}+1 + \frac{b e^\ve}{a}-1 \right)a\pio e^{-\ve} \ud x\\
   =&\int \left(\frac{b\rho}{a\pio e^{-V_\eps}} \log \frac{b\rho e^{V_\eps}}{a\pio} - \frac{b\rho}{a\pio e^{-\ve}}+1 \right) a\pio e^{-\ve} \ud x + b - ab=:I +b-ab,
\end{align*}
where we used $\rho, \pio\in\sP(\bR^n)$ in the second equality. Since $x\log x - x +1 \geq 0$, we have
\begin{align*}
    I\leq e^{-\min|\ve|}\int \left(\frac{b\rho}{  e^{-V_\eps}} \log \frac{b\rho e^{V_\eps}}{a\pio} - \frac{b\rho}{  e^{-\ve}}+a\pio \right) \ud x = e^{-\min|\ve|}\int \frac{b\rho}{  e^{-V_\eps}} \log \frac{b\rho e^{V_\eps}}{a\pio} \ud x.
\end{align*}
Plugging this estimate into \eqref{entL}, we obtain
\begin{equation}
\begin{aligned}
   b\int \rho \log \frac{\rho}{\pie}\ud x =b \log a + b-ab + I \leq e^{-\min|\ve|}\int \frac{b\rho}{  e^{-V_\eps}} \log \frac{b\rho e^{V_\eps}}{a\pio} \ud x,
\end{aligned}    
\end{equation}
where we used $\log a+1-a\leq 0.$
Combining this with \eqref{tm_lhs}, we obtain
\begin{equation}
    \int \rho \log \frac{\rho}{\pie}\ud x \leq \frac{a}{b} e^{-\min|\ve|} \text{LHS} \leq \frac{1}{2\lambda_0}e^{\max|\ve|-\min|\ve|}\int \rho |\nabla \log \frac{\rho}{\pie}|^2 \ud x.
\end{equation}
This implies that $\pie$ satisfies LSI with the stated $\lambda$ value.
\end{proof}
 
Next we comment about the assumption underlying the above result. 
For what follows, we recall some further properties of $\pi$ satisfying LSI, mainly from 
\cite[Theorem 9.9]{villani2003topics}.
Note that we are dealing with the following two probability measures:
\begin{equation}\label{three.pi}
\pi_0 (= \lbar{\pi} = \text{wk-}\lim_\eps \pie) \sim e^{-U_0}
\,\,\,\text{and}\,\,\,
\pie \sim e^{-U_0 - V_\eps}.
\end{equation}
If $U_0$ is uniformly convex in the sense that
$D^2U_0 \gtrsim  I$, then $\pi_0$ satisfies LSI($\lambda_0$) for some $\lambda_0$.
As an explicit example, we can take $U_0\sim |x|^2$ and thus $\pi_0 \sim e^{-|x|^2}$ which is the Gaussian case.
By our assumption \eqref{Ue} on the form of $U_\eps$, the two measures $\pi_0$ \eqref{pi0} and $\pie$ \eqref{pie} are comparable to each other up to multiplicative constants. Hence by Lemma \ref{lem:LSI-single}, there is a single number $\lambda$ so that both $\pi_0$ and $\pie$ satisfy $\text{LSI}(\lambda)$.
In the other direction, if $\pi$ satisfies an LSI, then there is a constant $C$ such that
$
\int e^{C|x|^2}\pi(x) \ud x < \infty
$ -- see also \cite[Proposition 5.4.1]{bakry2014analysis}.
    
Next we present the following compactness result in the whole space. 
A self-contained proof is provided here but see \cite{Hooton1981} for more general statements of this sort.
\begin{lem}\label{lem_compact}
Let $\pi$ satisfy Log-Sobolev inequality \eqref{LSI},
then the following compact embedding holds in the whole space 
\begin{equation}\label{compact}
H^1(\pi) \hookrightarrow\hookrightarrow L^2(\pi).
\end{equation}
Precisely, for any $A<\infty$,
\[
\left\{u: \int u^2\pi\ud x < A \right\}\bigcap
\left\{u: \int |\nabla u|^2\pi\ud x < A \right\}
\quad \text{is compact in $L^2(\pi)$.}
\]
\end{lem}
  \begin{proof}
      First, it is obvious that
   $$\int |u|^2 \pi \ud x \leq \int |u|^2(1+|x|^2)^{\frac12}\pi \ud x,$$
   which implies $L^2(\pi(1+|x|^2)^{\frac12}) \hookrightarrow L^2(\pi).$

   Second,  
   we claim $\int |u|^2(1+x^2) \pi \ud x$ is also bounded. Indeed, by  \cite[eq.(21.3)]{villani2009optimal}, we have that
   \begin{equation}\label{LSI-t1}
   \int u^2 \log (u^2) \pi \ud x \leq \frac{2}{\lambda} \int  |\nabla u|^2 \pi \ud x 
   + \left(\int  u^2 \pi \ud x\right)\log \left(\int  u^2 \pi \ud x\right) < \8.
   \end{equation}
   Then by inequality $a\cdot b \leq a(\log a-1) + e^b, a>0, b\in \bR$, we derive for any $C$ that
   \begin{equation} \label{est_x2}
   \begin{aligned}
       \int |u|^2(1+|x|^2) \pi \ud x
       &= \frac{1}{C}\int u^2 C(1+|x|^2)\pi \ud x \\
       &\leq \frac{1}{C}\int u^2 \Big(\log(u^2) -1\Big)\pi \ud x + \frac{1}{C}\int e^{C(1+|x|^2)}\pi \ud x\\
       &\leq \frac{2}{C\lambda}\left[\int |\nabla u|^2\pi \ud x+Z\log Z\right]
       +\frac{1}{C}\int e^{C(1+|x|^2)}\pi \ud x,
   \end{aligned}
   \end{equation}
   where in the second inequality we have used \eqref{LSI-t1}.
   We conclude by the comment before this lemma that 
   as $\pi$ satisfies LSI, there is a $C$ such that
   $\displaystyle \int e^{C(1+|x|^2)}\pi \ud x<\8$. 

   Third, from the boundedness of $\displaystyle \int |u|^2(1+|x|^2) \pi \ud x$, we obtain
   \begin{align*}
       \int_{B_R^c} |u|^2(1+|x|^2)^{\frac12} \pi \ud x \leq \frac{1}{\sqrt{1+R^2}}\int_{B_R^c} |u|^2(1+|x|^2)  \pi \ud x \leq \frac{C(\|u\|)}{\sqrt{1+R^2}},
   \end{align*}
   which goes to zero as $R\to+\8.$
   On the other hand, in $B_R$, one can apply the usual compact embedding in compact domain $H^1(\pi)\equiv H^1\hookrightarrow\hookrightarrow L^2 \equiv L^2((1+x^2)^{\frac12}\pi)$
   to obtain that for any fixed $R$, there exists a subsequence $u_{R,i}$ strongly converges to some $u^*$ in $L^2((1+x^2)^{\frac12}\pi, B_R)$.
   Therefore, by diagonal argument, there exists a subsequence (without relabel) $u_{i}$ converges almost everywhere to some $u^*$. By Fatou's Lemma, we have 
   $\displaystyle \int {u^*}^2(1+|x|^2)^{\frac12}\pi\ud x\leq \liminf_i\int u_i^2(1+|x|^2)^{\frac12}\pi\ud x < \infty$ and hence
   $u^*\in L^2((1+|x|^2)^{\frac12}\pi).$ From this, we can deduce that $u_i\rightarrow u^*$ in 
   $L^2((1+|x|^2)^{\frac12}\pi)$.
   This, together with the first step, implies
   $$H^1(\pi) \hookrightarrow\hookrightarrow L^2((1+|x|^2)^{\frac12}\pi) \hookrightarrow L^2(\pi).$$
  \end{proof}

\subsection{A priori Uniform estimates}
In order to study the asymptotic behavior as $\eps \to 0$, we would first establish some a-priori estimates for our $\eps$-gradient flow system
\eqref{epsFP} (or \eqref{epsGF}). These would then 
give us some space-time compactness and convergence properties. These estimates are
variational in nature. They are more or less standard, but for completeness, we provide sufficient details, in particular, to demonstrate their uniformity in $\eps$.

Before presenting our analysis, we give some remarks about the well-posedness of
equation \eqref{mainFP}. Given that its solution represents probability density, the most natural space is $L^1$ for probability measures. Well-posedness in this space can be established using Green's function. We refer to \cite[Theorem 2.1, Appendix A]{Tamura87} for their precise estimates. 
For this part, we simply treat $\eps$ as a fixed parameter. In order to investigate
long time behavior, we would further impose finite entropy (with respect to the
underlying invariant measure $\pie$) and finite second moment. This will be made clear
in Section \ref{longtime}. However, the technique to analyze the limit $\eps\to0$ is exclusively \emph{variational}. The most natural space for this is \mbox{(weighted-)}$L^2$ and 
hence the forthcoming computations. We believe that with more advanced techniques, we
can still handle initial data with weaker integrability but this would inevitably involve
some initial time transient behavior. We opt not to digress
in this direction as it is not the main purpose of our work. 
See also Remark \ref{rem:highreg}. With that, we now proceed to our $L^2$-analysis.

We notice that singular coefficients can appear whenever we differentiate terms involving the fast variable $\frac{x}{\eps}$.   On the other hand, the function $\displaystyle f^\eps_t := \frac{\rho^\eps_t}{\pi_\eps}$ is expected to have better regularities as its evolutionary equation is of variational form. To this end, 
we re-write \eqref{epsFP} as
\begin{equation}
\pt_t \rhoe = \nabla \cdot \bbs{\pi_\eps B_\eps^{-1} \nabla \frac{\rhoe}{\pi_\eps} + \rhoe B_\eps^{-1}\nabla W*\rhoe   },
\end{equation}
which, in terms of $f^\eps_t$ becomes
\begin{align}\label{backward}
\pt_t f^\eps_t = \frac{1}{\pi_\eps} \nabla \cdot \bbs{   \pi_\eps B_\eps^{-1}  \nabla f^{\eps}_t + f^\eps_t   \pi_\eps B_\eps^{-1}    \nabla W*(f^\eps_t\pie)  }.
\end{align}
Note that the operator $L_\eps(\cdot):= \frac{1}{\pi_\eps} \nabla \cdot \bbs{ \pi_\eps B_\eps^{-1} \nabla (\cdot)}$ is self-adjoint and negative definite with respect to
$L^2(\pi_\eps)$. More precisely,
\begin{eqnarray}
& \la L_\eps u, v \ra_{\pi_\eps} = \la u, L_\eps v \ra_{\pi_\eps}, \quad \forall u, v \in L^2(\pi_\eps),\\
& \text{and} \quad
\la L_\eps u, u \ra_{\pi_\eps} = -\int \la\nabla u, \pie B_\eps^{-1}\nabla u\ra\ud x < 0,
\quad\text{for $u\in L^2(\pie), \neq 0$.}
\end{eqnarray}
In the above and what follows, we will use $\la \cdot, \cdot \ra_{\pie}$ to denote
the $\pi_\eps$-weighted $L^2$-inner product,
$\displaystyle \la \cdot, \cdot \ra_{\pie} 
:= \int u(x)v(x)\pie(x)\ud x$ and $\|\cdot\|_{\pie}$ as the $\pie$-weighted $L^2$-norm. Thus we can resort to variational techniques to analyze equation \eqref{backward}.
On the other hand, the nonlinear term in \eqref{backward} is in convolution form and thus can be controlled via  Young's convolution inequality. 
 
With that, we present our result on the uniform estimates for $f^\eps_t$.
We recall again the assumptions in Section \ref{sec:assumptions}, in particular
the boundedness and uniform ellipticity of the matrix $B$ \eqref{B.est}, 
the form of $\pie$ \eqref{pie}, 
and the bounded and Lipschitz properties of the interacting kernel $W$ \eqref{W.est}.
We will also use the convention about the symbol $C$ for some generic constant. 
In fact, in this section, 
$\alpha_B, \beta_V, \gamma_W$ and $\delta_V$ are treated as some arbitrary constants.
Their precise values do not play an important role here.

\begin{lem}\label{lem_reg}
Let $f^\eps_0$ be the initial data for \eqref{backward}. We define,
\begin{eqnarray}
A_0 & := & \sup_{\eps>0}\int (f_0^\eps)^2\pi_\eps\ud x,\label{A0}\\
B_0 & := & \sup_{\eps>0}\int 
\langle \nabla f_0^\eps, B_\eps^{-1}\pi_\eps\nabla f_0^\eps\rangle\ud x.
\label{B0}
\end{eqnarray}
Let $0 < T < \infty$. Then we have the following statements.
\begin{enumerate}

\item If $A_0 < \infty$, then 
$f^\eps \in L^\infty((0,T);L^2(\pie))\bigcap L^2((0,T);H^1(\pie))$. 
In particular, for all $0 < t < T$, we have the following estimates, 
\begin{equation}\label{L2SpaceEst}
\|f_t^\eps\|^2_{\pi_\eps} \leq A_0 e^{Ct}, \quad \int_0^t \int \langle \nabla f_s^\eps, B_\eps^{-1}\pi_\eps \nabla f_s^\eps
\rangle \ud x \ud s \leq A_0(1+e^{Ct}).
\end{equation}

\item If $B_0 < \infty$ (which by $\displaystyle \int f_0\pie\ud x = 1$ and Poincare inequality implies 
$A_0 < \infty$), then
$$
f^\eps \in L^\infty((0,T);H^1(\pie))\bigcap H^1((0,T);L^2(\pie)).
$$
In particular, for all $0 < t < T$, we have the following estimates
\begin{equation}\label{f_time}
\int \langle \nabla f_t^\eps, B_\eps^{-1}\pi_\eps \nabla f_t^\eps
\rangle \ud x \leq C_t, \quad 
\int_0^t\int (\partial_s f_s^\eps)^2\pi_\eps \ud x \ud s\leq C_t.
\end{equation}
In addition, $f_t^\eps$ is $\frac12$-H\"{o}lder continuous in time with the following
estimate: 
\begin{equation}\label{f_Holder_time}
\|f_t^\eps-f^\eps_s\|_{\pi_\eps}
\leq C_T |t-s|^\frac12\quad\text{for any $0\leq s \leq t \leq T$.}
\end{equation}

\end{enumerate}
\end{lem}

\begin{proof} 
  As a preliminary, note that
\[
\pt_t f_t^\eps = B_\eps^{-1}:D^2f_t^\eps + 
\frac{1}{\pie}\big\langle\nabla(B_\eps^{-1}\pie),\nabla f_t^\eps\big\rangle
+\langle\nabla f_t^\eps, B_\eps^{-1}\nabla W\ast(f_t^\eps\pie)\rangle
+\frac{1}{\pie}f_t^\eps\nabla\Big(\pie B^{-1}_\eps\nabla W\ast(f_t^\eps\pie)\Big)
.
\]
Hence by maximum principle, we have $f^\eps_t > 0$ for $t > 0$. Furthermore, we have for all $t>0$ that
$\displaystyle \|\rho_t^\eps\|_{L^1} = \int \rho_t^\eps\ud x = 1$ so that 
$\displaystyle \| f^\eps_t\pie\|_{L^1} = \int f_t^\eps\pie\ud x = 1$. This leads to that
\begin{align}\label{conv1}
\|\nabla W * (f_t^\eps \pi_\eps)\|_{L^\8} \leq \|\nabla W\|_{L^\8}\|f_t^\eps \pi_\eps\|_{L^1} \leq \delta_W.
\end{align}

With the above, first, multiply both sides of \eqref{backward} by $f^\eps_t \pie$ and compute
\begin{eqnarray}\label{f_p1}
\frac{\ud}{\ud t}\frac12||f_t^\eps||^2_{\pi_\eps}
&=& \int f_t^\eps\partial_t f_t^\eps \pi_\eps\ud x
\\
&=& -\int \langle \nabla f_t^\eps, B_\eps^{-1}\pi_\eps \nabla f_t^\eps
\rangle \ud x + \int \la \nabla f_t^\eps, f_t^\eps B_\eps^{-1}\pi_\eps  \nabla W*(f_t^\eps \pi_\eps) \ra \ud x. \nonumber
\end{eqnarray}
Using \eqref{conv1}, the above becomes  
\begin{align}\label{f_p1_tem}
\frac{\ud}{dt}\frac12\|f_t^\eps\|^2_{\pi_\eps}\lesssim 
-\frac12\int \langle \nabla f_t^\eps, B_\eps^{-1}\pi_\eps \nabla f_t^\eps
\rangle \ud x  + \frac12\delta_W\|f^\eps_t\|_{\pie}^2,
\end{align} 
Then Gr\"onwall's inequality gives
\begin{equation}
\|f_t^\eps\|^2_{\pi_\eps} \leq A_0 e^{\delta_Wt} \quad \text{for $0\leq t\leq T$}.
\end{equation}
Combining this with the time integration of \eqref{f_p1_tem},  we also have
\begin{equation}
\int_0^t \int \langle \nabla f_s^\eps, B_\eps^{-1}\pi_\eps \nabla f_s^\eps
\rangle \ud x \ud s \lesssim A_0(1+ e^{\delta_Wt}) \quad \text{for $0\leq t\leq T$}.
\end{equation}

Third, for higher regularity, 
 we compute
\begin{eqnarray}\label{f_p2}
&&\frac{\ud}{\ud t}\frac12\int
\langle \nabla f_t^\eps, B_\eps^{-1}\pi_\eps \nabla f_t^\eps\rangle
\ud x\\
&=& \int
\langle \nabla \partial_tf_t^\eps, B_\eps^{-1}\pi_\eps \nabla f_t^\eps\rangle
\ud x
= -\int
\partial_tf_t^\eps\nabla\cdot\big(B_\eps^{-1}\pi_\eps \nabla f_t^\eps\big)
\ud x\\
&=& -\int (\partial_tf_t^\eps)^2\pi_\eps \ud x +\int \pt_t f_t^\eps \nabla \cdot\bbs{  f^\eps_t B_\eps^{-1}\pi_\eps \nabla W*(f_t^\eps \pi_\eps) }  \ud x\quad\text{(by \eqref{backward})}\nonumber\\
&=&
-\int (\partial_tf_t^\eps)^2\pi_\eps \ud x 
-\int \pt_t \nabla f_t^\eps \cdot\bbs{  f^\eps_t B_\eps^{-1}\pi_\eps \nabla W*(f_t^\eps \pi_\eps) }  \ud x\nonumber\\
&=&
-\int (\partial_tf_t^\eps)^2\pi_\eps \ud x 
-\frac{d}{dt}\int \nabla f_t^\eps \cdot\bbs{  f^\eps_t B_\eps^{-1}\pi_\eps \nabla W*(f_t^\eps \pi_\eps) }  \ud x\nonumber\\
&& 
+\int \nabla f_t^\eps \cdot\bbs{  \pt_t f^\eps_t B_\eps^{-1}\pi_\eps \nabla W*(f_t^\eps \pi_\eps) }  \ud x
+\int \nabla f_t^\eps \cdot\bbs{  f^\eps_t B_\eps^{-1}\pi_\eps \nabla W*(\pt_t f_t^\eps \pi_\eps)}  \ud x.
\end{eqnarray}
Integrating in time from $0$ to $t$ and integrating by parts, 
we have
\begin{eqnarray}
\frac12\int
\langle \nabla f_t^\eps, B_\eps^{-1}\pi_\eps \nabla f_t^\eps\rangle
\ud x
&=&
\frac12\int
\langle \nabla f_0^\eps, B_\eps^{-1}\pi_\eps \nabla f_0^\eps\rangle
\ud x -\int_0^t\int (\partial_sf_s^\eps)^2\pi_\eps \ud x\ud s\nonumber\\
&&
+ I_{\text{b.c.}} + I_1 + I_2,
\label{H1.est.0}
\end{eqnarray}
where 
\begin{eqnarray}
I_{\text{b.c.}} &:=&
\int \nabla f_0^\eps \cdot \bbs{  f^\eps_0 B_\eps^{-1}\pi_\eps \nabla W*(f_0^\eps \pi_\eps) } \ud x -  \int    \nabla f_t^\eps \cdot \bbs{  f^\eps_t B_\eps^{-1}\pi_\eps \nabla W*(f_t^\eps \pi_\eps) } \ud x,\\
I_1 &:=&
\int_0^t \int \nabla f_s^\eps \cdot  \Big(\pt_s f^\eps_s B_\eps^{-1}\pi_\eps \nabla W*(f_s^\eps \pi_\eps)\Big) \ud x \ud s, \\
I_2 &:=&
\int_0^t \int   \nabla f_s^\eps \cdot  \Big(f^\eps_s B_\eps^{-1}\pi_\eps \nabla W*(\pt_s f_s^\eps \pi_\eps)\Big)\ud x \ud s.
\end{eqnarray}
The last three terms are analyzed as follows.
\begin{itemize}
\item
For $I_{\text{b.c.}}$ at time $s=0, t$, we use \eqref{conv1} and Young's  inequality to obtain
\begin{align}\label{ft_ibc}
I_{\text{b.c.}} \lesssim \delta_W(A_0 + B_0) + \frac{1}{4} \int
\langle \nabla f_t^\eps, B_\eps^{-1}\pi_\eps \nabla f_t^\eps\rangle
\ud x  + \delta_W \|f_t^\eps\|^2_{\pi_\eps}.
\end{align}

\item 
For $I_1$, we use \eqref{conv1} and Young's  inequality to obtain
\begin{align}\label{ft_i1}
I_1\lesssim 
\frac14 \int_0^t\int (\partial_s f_s^\eps)^2\pi_\eps \ud x
+ \delta_W \int_0^t \int \langle \nabla f_s^\eps, B_\eps^{-1}\pi_\eps \nabla f_s^\eps
\rangle \ud x \ud s. 
\end{align}

\item 
For $I_2$, we use  Young's convolution inequality to obtain
\begin{align}\label{conv2}
\|\nabla W * (\pt_s f_s^\eps \pi_\eps)\|_{L^\8_x} \leq \delta_W\int |\pt_s f_s^\eps| \pi_\eps \ud x \leq    \delta_W  \bbs{\int |\pt_s f_s^\eps|^2 \pi_\eps \ud x}^\frac12 \quad \text{for $0\leq s \leq t$},
\end{align}
where the last inequality is implied by the Cauchy–Schwarz inequality with $\| \pi_\eps\|_{L^1}=1$. Then we have
\begin{align*}
I_2 \lesssim &\delta_W\int_0^t \Big[\bbs{\int |\pt_s f_s^\eps|^2 \pi_\eps \ud x}^\frac12 \int \left|f^\eps_s \nabla f_s^\eps\right|\pi_\eps \ud x   \Big]  \ud s\\
\lesssim & \delta_W    \int_0^t \Big[    \bbs{\int |\pt_s f_s^\eps|^2 \pi_\eps \ud x}^\frac12  \bbs{\int  \langle \nabla f_s^\eps, B_\eps^{-1}\pi_\eps \nabla f_s^\eps
\rangle \ud x}^{\frac12}  \bbs{\int |f^\eps_s  |^2  \pi_\eps \ud x}^{\frac12}   \Big]  \ud s.
\end{align*}
By \eqref{L2SpaceEst} and Young's inequality, the above estimate for $I_2$ becomes
\begin{equation}\label{ft_i2}
I_2 \leq 
\frac14 \int_0^t\int (\partial_s f_s^\eps)^2\pi_\eps \ud x
+ C_t\int_0^t \int \langle \nabla f_s^\eps, B_\eps^{-1}\pi_\eps \nabla f_s^\eps
\rangle \ud x \ud s.
\end{equation}
\end{itemize}

Combining the above estimates for $I_{\text{b.c}}, I_1, I_2$ in \eqref{ft_ibc}, \eqref{ft_i1}, \eqref{ft_i2}, the time integration of \eqref{f_p2} leads to
\begin{equation}
\frac{1}{4} \int \langle \nabla f_t^\eps, B_\eps^{-1}\pi_\eps \nabla f_t^\eps
\rangle \ud x + \frac12 \int_0^t\int (\partial_s f_s^\eps)^2\pi_\eps \ud x \ud s
\leq  C_t + C_t\int_0^t \int \langle \nabla f_s^\eps, B_\eps^{-1}\pi_\eps \nabla f_s^\eps
\rangle \ud x \ud s . 
\end{equation}
This, together with Gr\"onwall's inequality, yields
\begin{equation}
\int \langle \nabla f_t^\eps, B_\eps^{-1}\pi_\eps \nabla f_t^\eps
\rangle \ud x,
\,\,\,\text{and}\,\,\,
\int_0^t\int (\partial_s f_s^\eps)^2\pi_\eps \ud x \ud s
\,\leq\,
C_t \quad \text{for $0\leq t\leq T$.}
\end{equation}

Statement \eqref{f_Holder_time} follows directly from \eqref{f_time} by means of Cauchy-Schwartz:
\begin{eqnarray*}
\int(f_t^\eps-f_s^\eps)^2\pie\ud x
=
\int\left(\int_s^t\partial_r f_r^\eps\ud r\right)^2\pie\ud x
\leq 
\int(t-s)\int_s^t(\partial_r f_r^\eps)^2\ud r\,\pie\ud x
\leq C_T(t-s).
\end{eqnarray*}
\end{proof}

For our application, we will also need some regularity
estimates for the time derivative of $f^\eps$. Define
$h_t^\eps := \partial_t f^\eps_t$. Then it satisfies  
\begin{equation}\label{TimeDerEqn}
\pt_t h^\eps_t  
 = \frac{1}{\pi_\eps} \nabla \cdot \Big(   \pi_\eps B_\eps^{-1}  \nabla h^{\eps}_t + h^\eps_t   \pi_\eps B_\eps^{-1}    \nabla W*(f^\eps_t\pie)  + f^\eps_t   \pi_\eps B_\eps^{-1}    \nabla W*(h^\eps_t\pie)  \Big).
\end{equation}
Exactly the same proofs as the estimates for $I_1$, $I_2$ in   \eqref{ft_i1}, \eqref{ft_i2} from the previous lemma
give  the following result.

\begin{lem}\label{lem_timeder_reg}
Let $h^\eps_0 = \partial_t f_t^\eps|_{t=0}$ be the initial data for \eqref{TimeDerEqn}. We define,
\begin{eqnarray}
C_0 & := & \sup_{\eps>0}\int (h_0^\eps)^2\pi_\eps\ud x
\,\,\left(=
\sup_{\eps>0}\int (\pt_t f_0^\eps)^2\pi_\eps\ud x
\right).\label{C0}
\end{eqnarray}
Let $0 < T < \infty$ be given. If $B_0,\,  C_0 < \infty$, then
$h^\eps \in L^\infty((0,T);L^2(\pie))\bigcap L^2((0,T);H^1(\pie))$.
In particular, for all $0 < t < T$, the following estimates hold,
\begin{equation}\label{L2SpaceEst_timeder}
\|h_t^\eps\|^2_{\pi_\eps} \leq C_t, \quad \int_0^t \int \langle \nabla h_s^\eps, B_\eps^{-1}\pi_\eps \nabla h_s^\eps
\rangle \ud x \ud s \leq C_t.
\end{equation}
Similarly, if 
\begin{eqnarray}
D_0 & := & \sup_{\eps>0}\int \langle \nabla h_0^\eps, \pie B_\eps^{-1}\nabla h_0\rangle \ud x
\,\,\left(=
\sup_{\eps>0}\int \langle \nabla \pt_t f_0^\eps, \pie B_\eps^{-1}\nabla \pt_t f_0\rangle \ud x
\right)<+\8,\label{D0}
\end{eqnarray}
then $h^\eps \in L^\infty((0,T);H^1(\pie))\bigcap H^1((0,T);L^2(\pie))$.
In particular, for all $0 < t < T$, we have the following estimates,
\begin{equation}\label{H1SpaceTimeEst_timeder}
\|\nabla h_t^\eps\|^2_{\pie},\quad
\int_0^t \int(\partial_s h_s^\eps)^2\pi_\eps\ud x\ud s
\leq C_t
\end{equation}
and
\begin{equation}\label{HolderTimeEst_timeder}
\|h_t^\eps-h_s^\eps\|_{\pi_\eps} \leq C_T|t-s|^\frac12,\quad\text{for $0<s<t<T$.}
\end{equation}
\end{lem}

We make the following same remark as \cite[Remark 3.6]{GaoYip}.
\begin{rem}\label{rem:highreg}
Note that currently our approach does require a high
degree of regularity for the initial data. Its existence and 
construction would require the characterization of precise oscillations of the 
solution which in principle can be done by considering second and higher order
cell problems.
However, we believe this requirement can be much relaxed by means of parabolic 
regularity. For example, if $A_0 < \infty$, then $f_t^\eps\in H^1(\bR^n)$ for 
some $t>0$ and if $B_0 < \infty$, then  $\partial_t f^\eps_t\in L^2(\bR^n)$
for some $t>0$. This can be iterated due to the variational structure of 
equation \eqref{backward}.
Alternatively, we can also opt to utilize some technical results similar to
\cite[p.14, steps (a-c)]{jordan1998variational} and
\cite[Proposition 4.4]{forkert2022evolutionary} in which the initial data
even belongs to $L^1(\bR^n)$. For simplicity, in this paper, 
we do not pursuit this route, as we feel it is beyond the scope of 
homogenization which is our key motivation.
\end{rem}

Combining the apriori estimates from Lemma \ref{lem_reg}--\ref{lem_timeder_reg} and the compactness property in Lemma \ref{lem_compact}, we have the following convergence results which will be used in the rest of 
this paper.
\begin{lem} \label{cor.f}
Let $f^\eps=\frac{\rho^\eps}{\pie}$ be the solution of \eqref{backward} on $(0,T)$.
Suppose $B_0 < \8$ and $D_0<\8$.  
Then the following statements hold.
\begin{enumerate}
\item
There is a subsequence $f^\eps$ 
and an $f\in L^2(0,T; L^2(\pi))$ such that
$f^\eps\longrightarrow f$ in $L^2(0,T; L^2 (\pi))$, i.e.,
\begin{equation}\label{L2SpaceTimeConv}
\int_0^T\int|f^\eps_t-f_t|^2\pi\ud x\ud t\rightarrow0
\end{equation}
and hence
\begin{equation}\label{ae.conv}
\int|f^\eps_t-f_t|^2\pi\ud x\rightarrow0
\,\,\, \text{for all $t\in[0,T]$}.
\end{equation}
Here $\pi$ can be either of $\pi_0$ or $\pie$ in \eqref{three.pi}
so that $\text{LSI}(\lambda)$ holds.

\item We further have $\pt_tf^\eps\longrightarrow \pt_tf$ in $L^2(0,T; L^2 (\pi))$, i.e.,
\begin{equation}\label{L2SpaceTimeConvHigh}
\int_0^T\int|\pt_tf^\eps_t-\pt_tf_t|^2\pi\ud x\ud t\rightarrow0
\end{equation}
and hence
\begin{equation}\label{ae.conv.high}
\int|\pt_tf^\eps_t-\pt_tf_t|^2\pi\ud x\rightarrow0
\quad \text{for all $t\in[0,T]$}.
\end{equation}

\item Let $\eta^\eps_t = \sqrt{f^\eps_t}$. Then
\begin{equation}\label{conv.sqrt.feps}
    \eta^\eps_t \longrightarrow \eta_t := \sqrt{f_t}\,\,\,\text{in $L^2(\pi)$.}
\end{equation}
\item The entropy functional converges:
\begin{equation}\label{conv.entropy.eps}
\lim_\eps \int f^\eps_t(\log f^\eps_t) \pi_\eps\ud x =
\int f_t(\log f_t) \lbar{\pi}\ud x.
\end{equation}
(This of course strengthens the lower-semicontinuity of the entropy functional 
-- see Lemma \ref{lem:E}(i).)
\end{enumerate}
\end{lem}
\begin{proof}
Statement \eqref{L2SpaceTimeConv} is a direct consequence of 
\eqref{f_time} and the compactness lemma just proved. The fact that \eqref{ae.conv} holds for all $t\in[0,T]$ is due to the equicontinuity of $f_t^\eps$ in time -- see \eqref{f_Holder_time}. 
Similarly, \eqref{L2SpaceTimeConvHigh} and \eqref{ae.conv.high}
follow from \eqref{H1SpaceTimeEst_timeder} and \eqref{HolderTimeEst_timeder}.

For \eqref{conv.sqrt.feps},
let $\delta >0$ be some positive number.
\begin{eqnarray*}
&& \int (\sqrt{f^\eps}-\sqrt{f})^2 \pi\ud x\\
&=&
\int_{\left\{x: f^\eps_t(x),\,\,f_t(x) \leq \delta\right\}} (\sqrt{f^\eps}-\sqrt{f})^2 \pi\ud x
+\int_{\left\{x: f^\eps_t(x),\,\,f_t(x) \geq \delta\right\}} (\sqrt{f^\eps}-\sqrt{f})^2 \pi\ud x\\
&=:& I_1 + I_2.
\end{eqnarray*}
The integral $I_1$ converges to zero by Lebesgue Dominated Convergence Theorem. For $I_2$,
by means of Mean Value Theorem, we have
\begin{eqnarray*}
\int_{\left\{x: f^\eps_t(x),\,\,f_t(x) \geq \delta\right\}} (\sqrt{f^\eps}-\sqrt{f})^2 \pi\ud x
\leq 
\int_{\left\{x: f^\eps_t(x),\,\,f_t(x) \geq \delta\right\}} 
\left(\frac{1}{2\sqrt{\delta}}\right)^2(f^\eps - f)^2 \pi\ud x
\longrightarrow0
\end{eqnarray*}
as  $\eps\rightarrow0$.

For \eqref{conv.entropy.eps}, let $\delta$ be similarly a positive number. Then
\begin{eqnarray*}
&&\int f^\eps_t(\log f^\eps_t) \pie - f_t(\log f_t) \lbar{\pi}\ud x\\
&=&
\int \big(f^\eps_t(\log f^\eps_t) - f_t(\log f_t)\big) \pie\ud x
+\int f_t(\log f_t)(\pie-\lbar{\pi})\ud x\\
&=&
\int_{\left\{x: f^\eps_t(x),\,\,f_t(x) \leq \delta\right\}}\big(f^\eps_t(\log f^\eps_t) - f_t(\log f_t)\big) \pie\ud x\\
&&
+\int_{\left\{x: f^\eps_t(x),\,\,f_t(x) \geq \delta\right\}}\big(f^\eps_t(\log f^\eps_t) - f_t(\log f_t)\big) \pie\ud x\\
&&
+\int f_t(\log f_t)(\pie-\lbar{\pi})\ud x\\
&=:& I_1 + I_2 + I_3.
\end{eqnarray*}
In the above, $I_3$ goes to zero as $\pie\rightharpoonup\lbar{\pi}$.
Integral $I_1$ goes to zero by Lebesgue Dominated Convergence Theorem again and note that 
$\pie\lesssim\lbar{\pi}$.
For $I_2$, by Mean Value Theorem again, we have
\begin{eqnarray*}
&&\int_{\left\{x: f^\eps_t(x),\,\,f_t(x) \geq \delta\right\}}\big(f^\eps_t(\log f^\eps_t) - f_t(\log f_t)\big) \pie\ud x\\
&\lesssim &
\int_{\left\{x: f^\eps_t(x),\,\,f_t(x) \geq \delta\right\}}
(|\log f^\eps_t| + |\log f_t|+1)\big|f^\eps_t-f_t\big| \pie\ud x\\
& \lesssim &
\left(\int ({f^\eps_t}^2 + {f_t}^2+1) \pie\ud x\right)^\frac12
\left(\int \big|f^\eps_t-f_t\big|^2 \pie\ud x\right)^\frac12
\longrightarrow_{\eps\rightarrow0} 0
\end{eqnarray*}
where we have used the fact that $|\log y| \lesssim |y|$ for $y \geq \delta$.
\end{proof}

Note that in general as $\pie$ does not converge strongly, we can only assert that 
$\rho^\eps_t\rightharpoonup\rho_t$.
In the following, we state two time continuity properties for $\rho^\eps_t$  and $\rho_t$
with respect to the Wasserstein space $(\sP_2(\mathbb R^n), W_2)$.
\begin{prop}[H\"{o}lder continuity in time]\label{lem:rhos}
Assume $E_\eps(\rho_0^\eps)<+\8$. For any $T>0$, let $\rhoe, t\in[0,T]$ be a solution to the $\eps$-gradient flow system \eqref{epsEDI}. 
Then it holds that
\begin{equation}\label{equi-w2}
W_2^2(\rho^\eps_t, \rho^\eps_s) \leq C_T|t-s| \quad \text{for $0\leq s\leq t\leq T$},
\end{equation}
where $W_2(\cdot, \cdot)$ is the standard $W_2$-distance.    
Consequently, there exists a subsequence 
$\rho^\eps$ and $\rho\in C^{0,\frac12}([0,T]; \sP_2(\bR^n))$ such that
\begin{equation}\label{convergence_W2}
 \rhoe \rightharpoonup \rho_t, \quad W_2^2(\rho_t, \rho_s) \leq C_T|t-s| \quad 
 \text{ for $\leq s\leq t\leq T$}.
\end{equation} 
\end{prop}
\begin{proof}
First, by \eqref{epsGFu}, $\rho^\eps_t$ for $t\in (0,T)$, satisfies  
$\pt_t \rho^\eps_t = -\nabla \cdot \bbs{\rhoe \be \nabla u^\eps_t}$ and hence the 
continuity equation 
$\pt_t\rho^\eps_t + \text{div}(\rho^\eps_t w^\eps_t)=0$ with velocity 
$w^\eps_t = \be \nabla u^\eps_t$.
From \eqref{epsEDI}, \eqref{psi-eqn}, and $E_\eps(\rho_0^\eps)<+\8$, we have for any  $0\leq s\leq t\leq T$ that
\begin{equation}
\int_s^t \int\frac12 \la \nabla u^\eps_\tau, \be \nabla u^\eps_\tau  \ra \rho^\eps_\tau \ud x \ud \tau
=
\int_s^t \psi_\eps( \rho^\eps_\tau, \pt_\tau \rho^\eps_\tau) \ud \tau <+\8.
\end{equation}
Hence we can invoke \cite[Theorem 17.2]{ABS} to conclude that
\begin{equation}\label{eq3.49}
\begin{aligned}
W_2^2(\rho^\eps_t, \rho^\eps_s) \leq |t-s| \int_s^t \int |w^\eps_\tau|^2 \rho^\eps_\tau \ud x \ud \tau =& |t-s| \int_s^t \int |\be \nabla u^\eps_\tau|^2 \rho^\eps_\tau \ud x \ud \tau\\
\lesssim & |t-s|\int_s^t \int \la  \nabla u^\eps_\tau, \be  \nabla u^\eps_\tau \ra \rho^\eps_\tau \ud x \ud \tau=C_T|t-s|,
\end{aligned}
\end{equation}
giving the equi-continuity of $\rhoe$ in time in $(\sP(\mathbb R^n), W_2)$.

    Second, based on the regularity estimate for $f^\eps_t$ in Lemma \ref{lem_reg} and estimate \eqref{est_x2}, we know $\int |x|^2 (f^\eps_t)^2\pie\ud x<C.$ Then by Cauchy–Schwarz inequality,
  \begin{align*}
   \int |x|^2\rho^\eps_t\ud x
  = \int |x|^2 f^\eps_t\pie\ud x
  = \int |x| f^\eps_t\sqrt{\pie}|x|\sqrt{\pie}\ud x
  \leq \sqrt{\int |x|^2 (f^\eps_t)^2\pie\ud x}\sqrt{\int |x|^2\pie\ud x} < C_T.
  \end{align*}
  This uniform boundedness of the second moments implies the tightness of $\rhoe$, and thus by Prokhorov's theorem, we have the narrowly convergence of 
  \begin{eqnarray}
      \rho_t^\eps \rightharpoonup \rho_t, \quad \rho_s^\eps \rightharpoonup \rho_s.
  \end{eqnarray}
  By the Fatou lemma for the second moment of $\rho_t^\eps \in \sP_2(\bR^n)$, we also know $\int |x|^2\rho_t \ud x<+\8$ and the limit $\rho_t \in \sP_2(\bR^n)$.
 (Note that by abuse of notation, we use $\rho_t$ to represent its density
 which indeed exists and is given by $f_t\lbar{\pi}=\lim_\eps f^\eps_t\pi_\eps$.)
 
Third,
via
  the Kantoravich-Rubinstein Duality formulation for $W_2$ distance, we know for $\phi,\psi\in \text{Lip}_b(\bR^n), \phi(x)+\psi(y) \leq |x-y|^2$, 
  \begin{align*}
     \int \phi(x) \rho_t^\eps(x) \ud x + \int \psi(y) \rho_s^\eps(y) \ud y \leq W_2(\rho_t^\eps, \rho_s^\eps). 
  \end{align*}
From the narrowly convergence of $\rho^\eps$ in second step, we have
\begin{align*}
   \int \phi(x) \rho_t(x) \ud x + \int \psi(y) \rho_s(y) \ud y = \lim_{\eps\to 0}\int \phi(x) \rho_t^\eps(x) \ud x + \int \psi(y) \rho_s^\eps(y) \ud y \leq \liminf_{\eps\to 0} W_2(\rho_t^\eps, \rho_s^\eps).
\end{align*}
Then taking the supremum with respect to $\phi,\psi\in \text{Lip}_b(\bR^n), \phi(x)+\psi(y) \leq |x-y|^2$, we have
\begin{eqnarray}
    W_2(\rho_t, \rho_s) \leq \liminf_{\eps\to 0} W_2(\rho_t^\eps, \rho_s^\eps) \leq C_T|t-s|^{\frac12},
\end{eqnarray}
where we used the H\"older estimate in the first step.
\end{proof}

  By \cite[Lemma 17.8]{ABS} and \eqref{eq3.49}, we have that for $\eps>0$, $\rho^\eps$ is in fact 
\emph{absolutely continuous (AC) in time} in $(\sP_2(\bR^n),W_2)$. Our next result establishes the same property for $\rho$. This is an improvement of the 
$\frac12$-H\"{o}lder regularity in time and is consistent with the fact that $\rho$ satisfies some underlying continuity equation.

\begin{prop}[Absolute continuity in time]\label{lem:acwass} 
Let $\rho^\eps:[0,T]\to\sP_2(\mathbb R^n)$ be the solution to \eqref{epsFP} and $\rho:[0,T]\to\sP_2(\mathbb R^n)$ be the limit obtained in Lemma \ref{lem:rhos}, then the limit $\rho$ is a absolute continuous curve with respect to $W_2$ metric, denoted as $\rho\in AC([0,T];\sP_2(\mathbb R^n)).$
\end{prop}
\begin{proof}
    First, recall \eqref{epsFP} or \eqref{cont.eqn} satisfied by the solution curve $\rhoe$, 
    $$\pt_t \rho^\eps_t = -\nabla \cdot \bbs{\rhoe w_t^\eps}, \quad w_t^\eps := -B_\eps^{-1} (\nabla\log\frac{\rhoe}{\pi_\eps}+\nabla W*\rhoe)$$
    we have
\begin{equation}
\int_0^t \int |w_\tau^\eps|^2 \rho^\eps_\tau \ud x \ud \tau \lesssim \int_0^t\psi_\eps(\rho_\tau^\eps,\pt_\tau \rho_\tau^\eps)\ud \tau<+\8,
\end{equation}
where we used definition of $\psi_\eps$ in \eqref{psi}.

    Second,  recall  Lemma \ref{lem_reg} and Corollary \ref{cor.f} implies  $f^\eps\longrightarrow f$ in $L^2(0,T; L^2(\pi_0))$. Moreover, Upon setting 
    $\eta^\eps = \sqrt{f^\eps}$, we have
\begin{eqnarray*}
\8&>&
\frac12 \iint f_\eps \Big\la \nabla\Big(\log f_\eps + W*(f_\eps \pi_\eps)\Big),  \pi_\eps \be \nabla\Big(\log f_\eps + W*(f_\eps \pi_\eps)\Big)\Big\ra \ud x\\
&=&2 \iint \Big\langle
\Big(\nabla \eta^\eps + \frac12\eta^\eps \nabla W*((\eta^\eps)^2 \pi_\eps)\Big),  
\pi_\eps \be 
\Big(\nabla \eta^\eps + \frac12\eta^\eps \nabla W*((\eta^\eps)^2 \pi_\eps)\Big)
\Big\rangle \ud x\\
&=&
2 \iint \Big\langle
\nabla \eta^\eps,  
\pi_\eps \be 
\nabla \eta^\eps 
\Big\rangle \ud x
+2\iint \Big\langle
 \nabla \eta^\eps ,  
\pi_\eps \be 
\Big(\eta^\eps \nabla W*((\eta^\eps)^2 \pi_\eps)\Big)
\Big\rangle \ud x\\
&&+
\frac12 \iint \Big\langle
\Big(\eta^\eps \nabla W*((\eta^\eps)^2 \pi_\eps)\Big),  
\pi_\eps \be 
\Big(\eta^\eps \nabla W*((\eta^\eps)^2 \pi_\eps)\Big)
\Big\rangle \ud x.
\end{eqnarray*}
  Since the third term in the last equality is positive and the second term has a lower bound
\begin{align*}
   \iint \Big\langle
\Big(\nabla \eta^\eps ,  
\pi_\eps \be 
\Big(\eta^\eps \nabla W*((\eta^\eps)^2 \pi_\eps)\Big)
\Big\rangle \ud x \geq - \delta \iint \Big\langle
\nabla \eta^\eps,  
\pi_\eps \be 
\nabla \eta^\eps 
\Big\rangle \ud x - C(\delta, \delta_W),
\end{align*}
we know 
$$\iint \big\langle
\nabla \eta^\eps,  
\overline{\pi} \be 
\nabla \eta^\eps 
\big\rangle \ud x\leq \iint \big\langle
\nabla \eta^\eps,  
\pi_\eps \be 
\nabla \eta^\eps 
\big\rangle \ud x<+\8.$$
Here we used the fact that by \eqref{pie} and \eqref{pi.ave}, $\frac{\pie}{\overline{\pi}}$ has lower and upper bounds. Thus we also have 
 $\sqrt{f^\eps}\longrightarrow \sqrt{f}$ in $L^2(0,T; L^2(\overline{\pi}))$.

 Third, recast the weak form of continuity equation as
 \begin{align*}
     -\iint \rhoe \pt_t\varphi \ud x \ud t =& \iint \rhoe w^\eps_t \cdot \nabla \varphi \ud x\ud t\\
     =&\iint f^\eps_t w^\eps_t \pie \cdot \nabla \varphi \ud x\ud t\\
     =&\iint \sqrt{f^\eps_t} \bbs{\frac{\sqrt{f^\eps_t}w^\eps_t \pie}{\overline{\pi}}}\overline{\pi} \cdot \nabla \varphi \ud x\ud t.
 \end{align*}
   Since  $$\iint \bbs{\frac{\sqrt{f^\eps_t}w^\eps_t \pie}{\overline{\pi}}}^2\overline{\pi} \lesssim \iint |w_t^\eps|^2 \rhoe \ud x \ud t<+\8,$$
   we know there exists $\Phi \in L^2(0,T; L^2(\overline{\pi}))$ such that  
   $$\frac{\sqrt{f^\eps}w^\eps \pie}{\overline{\pi}} \wra \Phi \quad \text{ in }L^2(0,T; L^2(\overline{\pi}))$$
   and
   \begin{equation}
       \iint \Phi^2_t \overline{\pi} \ud x \ud t \leq \liminf \iint \bbs{\frac{\sqrt{f^\eps_t}w^\eps_t \pie}{\overline{\pi}}}^2\overline{\pi} \ud x \ud t =\liminf \iint |w^\eps_t|^2 \rhoe \frac{\pie}{\overline{\pi}} \ud x \ud t<+\8.
   \end{equation}
   This, together the strong convergence of $\sqrt{f^\eps}$, allows one to take limit in the continuity equation to obtain
   \begin{align*}
       -\iint \rho_t \pt_t\varphi \ud x \ud t = \iint \sqrt{f_t} \Phi_t \overline{\pi} \cdot \nabla \varphi \ud x \ud t= \iint f_t \overline{\pi} \frac{ \Phi_t}{\sqrt{f_t}}  \cdot \nabla \varphi \ud x \ud t=\iint \rho_t  \frac{ \Phi_t}{\sqrt{f_t}} \cdot \nabla \varphi \ud x \ud t.
   \end{align*}
   Notice the limiting velocity field $\frac{ \Phi_t}{\sqrt{f_t}}$ satisfies
   $$\iint \frac{\Phi_t^2}{f_t} \rho_t\ud x \ud t=\iint \Phi^2_t \overline{\pi} \ud x \ud t <+\8.$$
   We can apply \cite[Theorem 2.15]{AGHB} to conclude that $\rho$ is absolutely continuous in time with respect to the $W_2$-metric.
\end{proof}

\section{Passing limit in EDI formulation of $\eps$-gradient flow}
\label{epsEDI-0EDI}
In this section, we prove that the EDI formulation \eqref{epsEDI} of 
the $\eps$-gradient flow \eqref{epsGF} converges to the limiting EDI
\eqref{0EDI}, which is equivalent to \eqref{0GF}. Thus it yields the homogenized limiting Wasserstein gradient flow. In essence, in Theorem \ref{thm1}, we need to prove \emph{sharp} lower bounds for the 
three functionals appearing on the left-hand-side of \eqref{epsEDI}. 
Recall the definitions of $\bar{E}, \psi, \psi^*$ in Section \ref{main}. 
The main result is stated in the following.
\begin{thm}\label{thm1}
Suppose the initial data $\rho_0^\eps$ is
well-prepared in the sense of \eqref{id_con} with weak limit $\rho_0$.
Furthermore, the assumptions of Lemma \ref{cor.f} hold.
Then,
  
\begin{enumerate}[(i)]
\item there exists a subsequence $\rho^\eps$ and a $\rho\in C^{0,\frac12}([0,T]; \sP_2(\bR^n))\cap AC([0,T];\sP_2(\mathbb R^n))$ such that $\rho^\eps\rightharpoonup\rho$;
\item for all $t\in[0,T]$, 
the lower bound for free energy holds
\begin{equation}\label{lower1}
\liminf_{\eps \to 0} E_\eps(\rhoe) \geq \lbar{E}(\rho_t);
\end{equation}
\item for all $t\in[0,T]$, the lower bound for the dissipation on the cotangent plane holds
\begin{equation}\label{lower2}
\liminf_{\eps \to 0} \int_0^t \psi^*_\eps\left(\rho^\eps_\tau, - \frac{\delta E_\eps}{\delta \rho}(\rho^\eps_\tau)\right) \ud \tau \geq  
\int_0^t \psi^*\left(\rho_\tau, - \frac{\delta\lbar{E}}{\delta \rho}(\rho_\tau)\right) \ud \tau;
\end{equation}
\item for all $t\in[0,T]$, the lower bound for the dissipation on the tangent plane holds
\begin{equation}\label{lower3}
\liminf_{\eps \to 0} \int_0^t \psi_\eps(\rho^\eps_\tau, \pt_\tau \rho^\eps_\tau) \ud \tau \geq  \int_0^t \psi(\rho_\tau, \pt_\tau \rho_\tau) \ud \tau.
\end{equation}
\end{enumerate}
\end{thm}

As mentioned before, our approach relies on the idea of
convergence of functionals in a variational setting. In particular, we make 
use of the following classical result of $\Gamma$-convergence \cite{braides2002gamma, braides2006handbook, dal2012introduction}.
\begin{thm}[$\Gamma$-conv]\label{GammaMainThm}
Let $\Omega$ be a bounded open domain of $\mathbb R^n$ and 
$A_\eps(x) = A(x, \frac{x}{\eps})$ be a symmetric positive definite
matrix, $1$-periodic in the second variable $y=\frac{x}{\eps}$. Consider the functional
\begin{equation}
{\mathcal F}_\eps(v) =\int_\Omega
\left\la A\left(x,\frac{x}{\eps}\right)\nabla v, \nabla v\right\ra\ud x,
\quad v\in H^1_0(\Omega) + w
\end{equation}
where $w\in H^1(\Omega)$ is given. Then ${\mathcal F}_\eps$ $\Gamma$-converges 
in $L^2(\Omega)$ to the following functional
\begin{equation}
{\mathcal F}(v) =\int_\Omega
\left\la \lbar{A}(x)\nabla v, \nabla v\right\ra\ud x,
\quad v\in H^1_0(\Omega) + w.
\end{equation}
In detail, 
\begin{enumerate}
\item for any $v_\eps\in H^1_0(\Omega) + w$ that converges to 
$v\in H^1_0(\Omega) + w$ in $L^2(\Omega)$, it holds that
\begin{equation}
\liminf_{\eps\to0}{\mathcal F}_\eps(v_\eps) \geq {\mathcal F}(v);
\end{equation}
\item for any $v\in H^1_0(\Omega) + w$, there exists 
$v_\eps\in H^1_0(\Omega) + w$ that converges to $v$ in $L^2(\Omega)$, such 
that 
\begin{equation}
\lim_{\eps\to0}{\mathcal F}_\eps(v_\eps) = {\mathcal F}(v).
\end{equation}
\end{enumerate}
Furthermore, the effective matrix $\lbar{A}$ can be found by the following
variational formula: for any $p\in\mathbb R^n$,
\begin{equation}\label{eff.homog.coeff}
\big\la \lbar{A}(x)p, p\big\ra 
= \inf\left\{\int_{\mathbb T^n} 
\left\la A\left(x,y\right)(\nabla v+p),\, (\nabla v + p)\right\ra\ud y,
\quad v\in H^1(\mathbb T^n)\right\}.
\end{equation}
\end{thm}
We make the following remarks about the above result.
\begin{rem}\hspace{10pt}
\begin{enumerate}
\item 
The above result is quite classical. We refer to \cite{braides2002gamma, braides2006handbook, dal2012introduction} for more detailed explanations and overall exposition. We would like to point out that traditionally, the result was first proved for periodic coefficient, $A_\eps(\cdot) = A(\frac{\cdot}{\eps})$ -- see for example \cite[Theorems 4.1, 4.4]{marcellini1978periodic} and 
also \cite{muller1987homogenization} for a more constructive proof. To go from 
$A_\eps(\cdot) = A(\frac{\cdot}{\eps})$ to 
$A_\eps(\cdot) = A(\cdot,\frac{\cdot}{\eps})$, one common strategy is to \emph{freeze} the 
slow variable -- in the current case $x$ -- as already revealed by the formula \eqref{eff.homog.coeff}. 
We refer to \cite[Corollary 6.1]{weinan1991class} and its proof for explicit steps of how to accomplish this.

\item As an application, we will apply the above result to the case 
$A_\eps(x) = \pie(x)B_\eps^{-1}$, or
\begin{equation}
A(x,y) = D(x,y)\,\,(= \pi(x,y)B^{-1}(y))
\quad\text{(see \eqref{A.def}).}
\end{equation}
The resultant formula for $\lbar{A}(x)$ is given by \eqref{eff.A} in Appendix \ref{asym.exp}. In Appendix \ref{asym.exp}, we derive the same formula using asymptotic analysis.

\item Although the above classical $\Gamma$-convergence result is stated in bounded domain, it also holds true for $\bR^n$ in our current case. This can be seen from $A_\eps(x) = \pie(x)B_\eps^{-1}(x)$ and the weight $\pie$ decays uniformly (in $\eps$) at far field, which allows one to extend the lower and upper bound estimates to whole domain using cutoff techniques.
\end{enumerate}
\end{rem}

With that, we proceed to the proofs.

\subsection{Proof of statements (i) and \eqref{lower1}.} \label{sec4.1}
Statement (i) is the content of Lemma \ref{lem:rhos} and \ref{lem:acwass}.
Conclusion \eqref{lower1} follows directly
from \cite[Lemma 9.4.3]{AGS} which says that the entropy
functional is jointly lower-semicontinuous with respect to the weak
convergence of $\rhoe$ and $\pie$. 
(This is a more ``advanced'' statement than 
Lemma \ref{lem:E}(i) as both $\rhoe$ and the reference measure $\pie$ are changing.) 
In our case, it also follows simply
from the strong convergence of $f_t^\eps$ to $f_t$ in $L^2(\pie)$ so that we in fact have
continuity -- see \eqref{conv.entropy.eps}.
Again, the validity of the statement for all $t$ is due to the continuity
of $f_t$ in time -- see \eqref{f_Holder_time}.


\subsection{Proof of \eqref{lower2} (time independence case)}\label{sec4.2}
Recall 
\[\psi^*_\eps (\rhoe, - \frac{\delta E_\eps}{\delta \rho})=
\frac12 \int \rhoe \left\la \nabla ( \log\frac{\rhoe}{\pi_\eps}+  W*\rhoe), \be \nabla ( \log\frac{\rhoe}{\pi_\eps}+  W*\rhoe)  \right\ra \ud x.
\]
Introducing $\eta_\eps = \sqrt{f_\eps}$ where $f_\eps = \frac{\rhoe}{\pie}$. Then
$\eta_\eps\longrightarrow\bar{\eta}:=\sqrt{f}$ in $L^2(\bar{\pi})$ -- see \eqref{conv.sqrt.feps}. 
We compute,
\begin{eqnarray}
\8&>&
\psi^*_\eps (\rhoe, - \frac{\delta E_\eps}{\delta \rho})\nonumber\\
&=&
\frac12 \int f_\eps \Big\la \nabla\Big(\log f_\eps + W*(f_\eps \pi_\eps)\Big),  \pi_\eps \be \nabla\Big(\log f_\eps + W*(f_\eps \pi_\eps)\Big)\Big\ra \ud x\nonumber\\
&=&
2 \int \Big\langle
\Big(\nabla \eta^\eps + \frac12\eta^\eps \nabla W*((\eta^\eps)^2 \pi_\eps)\Big),  
\pi_\eps \be 
\Big(\nabla \eta^\eps + \frac12\eta^\eps \nabla W*((\eta^\eps)^2 \pi_\eps)\Big)
\Big\rangle \ud x\nonumber\\
&=&
2 \int \Big\langle
\Big(\nabla \eta^\eps 
+ \frac12\bar{\eta}\nabla W*(\bar{\eta}^2\bar{\pi}) 
\Big),\pi_\eps \be 
\Big(\nabla \eta^\eps 
+ \frac12\bar{\eta}\nabla W*(\bar{\eta}^2\bar{\pi}) 
\Big)
\Big\rangle \ud x\nonumber\\
&& +4\int \Big\langle
\Big(\frac12\eta^\eps\nabla W*((\eta^\eps)^2 \pi_\eps)
- \frac12\bar{\eta}\nabla W*(\bar{\eta}^2\bar{\pi}) 
\Big),\pi_\eps \be 
\Big(\nabla \eta^\eps 
+ \frac12\bar{\eta}\nabla W*(\bar{\eta}^2\bar{\pi}) 
\Big)
\Big\rangle \ud x\nonumber\\
&&
+\frac12 \int \Big\langle
\Big( \eta^\eps\nabla W*((\eta^\eps)^2 \pi_\eps)
-  \bar{\eta}\nabla W*(\bar{\eta}^2\bar{\pi})\Big)
,\pi_\eps \be 
\Big(\eta^\eps\nabla W*((\eta^\eps)^2 \pi_\eps)
- \bar{\eta}\nabla W*(\bar{\eta}^2\bar{\pi})\Big)
\Big\rangle \ud x.\nonumber\\
\label{key.psi}
\end{eqnarray}
For the second and third terms of the above, we recall
  $f_\eps \to f$ in $L^2(\bar{\pi})$, $  f_\eps \pie=\rho_\eps \wra \rho$ and $\pie \wra \bar{\pi}$, we have $f_\eps \pie \wra f\bar{\pi}=\rho$, i.e,
$$(\eta^\eps)^2\pie \wra \bar{\eta}^2 \bar{\pi}.$$ 
This, together with $\nabla W \in L^\8$, implies
\begin{align*}
    |\int\nabla W(x,y) ((\eta^\eps)^2\pie-\bar{\eta}^2 \bar{\pi}) \ud y| \to 0, \quad \text{ for a.e. }x\in \bR^n.
\end{align*}
Thus 
we have
\begin{align*}
  &\int  
\Big|\eta^\eps\nabla W*((\eta^\eps)^2 \pi_\eps)
- \bar{\eta}\nabla W*(\bar{\eta}^2\bar{\pi}) 
\Big|^2\pi_\eps \ud y \\
\lesssim& \int|\eta^\eps -\bar{\eta}|^2\pie \ud y + \int\Big|\nabla W*((\eta^\eps)^2 \pi_\eps)
-  \nabla W*(\bar{\eta}^2\bar{\pi}) 
\Big|^2\bar{\eta}^2\pie \ud y \to 0,
\end{align*}
where we used $\eta_\eps\longrightarrow\bar{\eta}$ in $L^2(\bar{\pi})$ and $\frac{\pie}{\bar{\pi}}$ has lower and upper bound.
Therefore, for the second and third term in \eqref{key.psi}, we have
\begin{eqnarray*}
&&\left|\int \Big\langle
\Big(\eta^\eps\nabla W*((\eta^\eps)^2 \pi_\eps)
- \bar{\eta}\nabla W*(\bar{\eta}^2\bar{\pi}) 
\Big),\pi_\eps \be 
\Big(\nabla \eta^\eps 
+ \frac12\bar{\eta}\nabla W*(\bar{\eta}^2\bar{\pi}) 
\Big)
\Big\rangle \ud x\right|\\
&\lesssim&
\|\eta_\eps - \bar{w}\|_{L^2(\bar{\pi})}\|\nabla \eta_\eps\|_{L^2(\bar{\pi})}
\longrightarrow0\\
&&
\left|\int \Big\langle
\Big(\eta^\eps\nabla W*((\eta^\eps)^2 \pi_\eps)
- \bar{\eta}\nabla W*(\bar{\eta}^2\bar{\pi})\Big)
,\pi_\eps \be 
\Big(\eta^\eps\nabla W*((\eta^\eps)^2 \pi_\eps)
- \bar{\eta}\nabla W*(\bar{\eta}^2\bar{\pi})\Big)
\Big\rangle \ud x\right|\\
&\lesssim&
\|\eta_\eps - \bar{\eta}\|_{L^2(\bar{\pi})}^2
\longrightarrow0.
\end{eqnarray*}
Hence, we will just concentrate on the first term in \eqref{key.psi}. To this end, we have the following lemma:
\begin{lem}\label{Gamma2}
Let $\lbar{V}(x):=\frac12\bar{\eta}\nabla W*(\bar{\eta}^2\bar{\pi})$. Then
for any $\eta^\eps\longrightarrow \bar{\eta}$ weakly in $H^1$, we have
\begin{multline}
\liminf_\eps\int \Big\langle
\pi_\eps(x) \be(x)\Big(\nabla \eta^\eps + \lbar{V}(x) 
\Big), 
\Big(\nabla \eta^\eps + \lbar{V}(x)
\Big)
\Big\rangle \ud x\\
\geq
\int \left\langle
\overline{A}(x)\Big(\nabla \bar{\eta} +\lbar{V}(x)
\Big), 
\Big(\nabla \bar{\eta}+\lbar{V}(x)
\Big)
\right\rangle\ud x,
\end{multline}
where $\overline{A}$ is given by the following cell problem:
\begin{equation}
\big\la \lbar{A}(x)p, p\big\ra 
= \inf\left\{\int_{\mathbb T^n} 
\left\la \pi(x,y)B_\eps(y)^{-1}(\nabla \varphi+p),\, (\nabla \varphi + p)\right\ra\ud y,
\quad \varphi\in H^1(\mathbb T^n)\right\}.
\end{equation}
\end{lem}

Accepting the above, we then have
\begin{eqnarray}
\liminf_\eps \psi^*_\eps (\rhoe, - \frac{\delta E_\eps}{\delta \rho})
&\geq&
2 \int \Big\langle\lbar{A}(x)
\Big(\nabla \bar{\eta}
+ \frac12\bar{\eta}\nabla W*(\bar{\eta}^2\bar{\pi}) 
\Big),
\Big(\nabla \bar{\eta} 
+ \frac12\bar{\eta}\nabla W*(\bar{\eta}^2\bar{\pi}) 
\Big)
\Big\rangle \ud x\nonumber\\
&=&
\frac12 \int f\Big\langle \lbar{A}(x)
\nabla\Big(\log (f)+ \nabla W*(f\bar{\pi})\Big), 
\nabla\Big(\log (f)+ \nabla W*(f\bar{\pi})\Big)
\Big\rangle \ud x\nonumber\\
&=&
\frac12 \int \Big\langle \frac{1}{\bar{\pi}}\lbar{A}(x)
\nabla\Big(\log (\frac{\rho}{\bar{\pi}})+  W*\rho\Big), 
\nabla\Big(\log (\frac{\rho}{\bar{\pi}})+ W*\rho)\Big)
\Big\rangle \rho \ud x\nonumber\\
&=&\psi^*(\rho, - \frac{\delta \lbar{E}}{\delta \rho}).
\label{psi-star-liminf}
\end{eqnarray}

\begin{proof}[Remark about proof of Lemme \ref{Gamma2} .]
We can basically follow the same strategy as in \cite[Proof of Lemma 2.1(b), p.196]{muller1987homogenization}. Note that the functional is convex in $p$ and hence it is sufficient
to consider the infimum over a unit cell $\mathbb{T}^n$ -- see \cite[Lemma 4.1, p.207]{muller1987homogenization} or \cite[Theorems 2.1]{marcellini1978periodic}.
To extend to the spatial inhomogeneous case, as mentioned, a common step is to freeze the slow variable
-- $x$ in the current case.
For this, we can follow the proof of \cite[Corollary 6.1]{weinan1991class}.
\end{proof}


\subsection{Proof of \eqref{lower3} (time independence case)}\label{sec4.3}
Here we establish 
\begin{equation}\label{psi.lsc0}
\liminf_{\eps\to0}\psi_\eps(\rho^\eps, s^\eps)
\geq
\psi(\rho, s)
\end{equation}
for any $\rho^\eps\wra\rho$ in $L^1(\Omega)$ and $s^\eps\wra s$ in 
$L^2(\Omega)$ with the property that
\[
f^\eps = \frac{\rho^\eps}{\pie} \longrightarrow 
f = \frac{\rho}{\lbar{\pi}}\,\,\,\text{in}\,\,\,L^2(\lbar{\pi}).
\]

Recall 
\begin{equation*}
\begin{aligned}
\psi(\rho,s) 
=&\sup_{\xi\in L^2(\bR^n)}\left\{
\int \xi s\ud x 
- \frac12\int \la\nabla\xi,\bar{B}^{-1}\nabla\xi \ra\rho\ud x
\right\}.
\end{aligned}
\end{equation*}
Let $\xi^\delta$ be a maximizing sequence, i.e
\[
\psi(\rho,s) 
=\lim_\delta
\left\{\int \xi^\delta s\ud x 
- \frac12\int \la\nabla\xi^\delta,\bar{B}^{-1}\nabla\xi^\delta \ra\rho\ud x\right\}.
\]
Now, for each $\delta$, choose an approximating sequence $\tilde\xi^{\delta,\eps}\wra{\xi}^\delta$ in $H^1$ and $\tilde\xi^{\delta,\eps}\longrightarrow{\xi}^\delta$ in $L^2$    such that 
\begin{align*}
 &\lim_{\eps \to 0} \frac12\int \la\nabla\tilde\xi^{\delta,\eps},{B}_\eps^{-1}\nabla\tilde\xi^{\delta,\eps}
\ra\rho^\eps\ud x =\lim_{\eps \to 0} \frac12\int \la\nabla\tilde\xi^{\delta,\eps},{B}_\eps^{-1}\pie\nabla\tilde\xi^{\delta,\eps}
\ra f^\eps\ud x\\
=&\frac12\int \la\nabla\xi^\delta,\lbar{A}\nabla\xi^\delta \ra f\ud x=\frac12\int \la\nabla\xi^\delta,\lbar{B}^{-1}\nabla\xi^\delta \ra\rho\ud x.    
\end{align*}
The existence of such a recovery sequence is essentially ensured by Theorem \ref{GammaMainThm}.   In order to
separate the dependence between $\be\pie$ and $f^\eps$, we refer to  \cite[Appendix B]{GaoYip} for a detailed construction. Although \cite{GaoYip} considers a bounded domain, we can extend the construction to the whole domain $\bR^n$ using a standard cutoff argument. This can be seen from (i) $A_\eps(x) = \pie(x)B_\eps^{-1}(x)$, 
$\lbar{A}(x)=\lbar{\pi}(x)\lbar{B}^{-1}$, and the weight $\pie$, $\lbar{\pi}$ decay (uniformly in $\eps$) at far field; and (ii) the equi-integrability for the constructed $\la\tilde\xi^{\delta,\eps},{B}_\eps^{-1}\pie\nabla\tilde\xi^{\delta,\eps}
\ra$ is uniform   on exhausting balls  $B_R$ as $R\to +\8$.

Hence, by the fact that $\tilde\xi^{\delta,\eps}\longrightarrow{\xi}^\delta$ in $L^2$ and $s^\eps \wra s$ in $L^2$, we have
\begin{equation*}
\begin{aligned}
&\int  \xi^\delta s\ud x
- \frac12\int \la\nabla \xi^\delta,\lbar{B}^{-1}\nabla\xi^\delta \ra\rho\ud x
\\
=&
\lim_{\eps\to0}
\left\{\int \tilde\xi^{\delta,\eps} s^\eps\ud x
- \frac12\int \la\nabla\tilde\xi^{\delta,\eps},{B}_\eps^{-1}\nabla\tilde\xi^{\delta,\eps}
\ra\rho^\eps\ud x
\right\}\\
\leq & \liminf_{\eps\to0}\left[
\sup_{\xi}\left\{\int \xi s^\eps\ud x
- \frac12\int \la\nabla\xi,{B}_\eps^{-1}\nabla\xi \ra\rho^\eps\ud x
\right\}\right]
\\
\leq &\liminf_{\eps\to0} \psi(\rho^\eps, s^\eps).
\,\,\,
\end{aligned}
\end{equation*}
Now taking the limit $\delta\to0$ gives
\[
\psi(\rho, s)\leq \liminf_{\eps\to0} \psi_\eps(\rho^\eps, s^\eps).
\]


\subsection{Proof of \eqref{lower2} and \eqref{lower3}: time dependent case}\label{sec4.4}
To extend the time independent case to the time dependent case and finish the proofs of lower bounds \eqref{lower2}\eqref{lower3}, we will make use of a
general $\Gamma$-$\liminf$ result as stated in \cite[Cor. 4.4]{stefanelli2008brezis}.
Specifically, let $H$ be a separable and reflexive Banach space, and
$g_n$, $g_\infty: (0,T)\times H\longrightarrow(-\infty, \infty]$ be 
such that $g_n(t,\cdot)$ and $g_\infty(t,\cdot): H\longrightarrow(-\infty, \infty]$ 
are convex and for all $u\in H$ and a.e. $t\in(0,T)$, the following holds:
\begin{equation}\label{wk.normal.int}
g_\infty(t,u) 
\leq 
\inf\left\{
\liminf_{n} g_n(t,u_n): u_n\wra u\,\,\,\text{in $H$}
\right\}.
\end{equation}
Then for $p\in[1, \infty]$, $u_n\wra u$ in $L^p(0,T;H)$ 
(weak-$*$ if $p=\infty$)
and $t\longrightarrow \max\{0, -g_n(t,u_n(t)\}$ uniformly integrable, 
we have,
\begin{equation}
\int_0^T g_\infty(t,u(t))\ud t
\leq 
\liminf_n \int_0^T g_n(t,u_n(t))\ud t.
\end{equation}
Note that the uniform integrability condition is automatically satisfied if
$g_n$ are non-negative, or bounded from below. See also the remark after 
Cor. 4.4 in \cite{stefanelli2008brezis}.

For \eqref{lower2}, we set $H=H^1(\Omega)$. Recall \eqref{key.psi} and consider
\begin{eqnarray*}
\psi^*_\eps (\rhoe, - \frac{\delta E_\eps}{\delta \rho}) & = & 
g^*_\eps(t,\eta) + I_{\text{err}},\\
\text{where}\quad
g^*_\eps(t,\eta) &:=&
2 \int \Big\langle
\Big(\nabla \eta 
+ \frac12\bar{\eta}\nabla W*(\bar{\eta}^2\bar{\pi}) 
\Big),\pi_\eps \be 
\Big(\nabla \eta 
+ \frac12\bar{\eta}\nabla W*(\bar{\eta}^2\bar{\pi}) 
\Big)
\Big\rangle \ud x,
\end{eqnarray*} 
and $I_{\text{err}}$ is the second and third terms in \eqref{key.psi} with limit being zero
due to strong convergence of $f_\eps$ to $f$.
Let
\begin{eqnarray*}
g^*_\infty(t,\eta) & = & 2 \int \Big\langle\lbar{A}(x)
\Big(\nabla \eta
+ \frac12\bar{\eta}\nabla W*(\bar{\eta}^2\bar{\pi}) 
\Big),
\Big(\nabla \eta 
+ \frac12\bar{\eta}\nabla W*(\bar{\eta}^2\bar{\pi}) 
\Big)
\Big\rangle \ud x.
\end{eqnarray*}
Then $g^*_\eps(t,\cdot)$ and $g^*_\infty(t,\cdot)$ are convex and
\eqref{wk.normal.int} holds true by the time independent version  \eqref{psi-star-liminf}. 
This concludes the lower bound \eqref{lower2}.

For \eqref{lower3}, we set $H = L^2(\Omega)$,
\begin{eqnarray*}
g_\eps(t,s) &:= &\psi_\eps(\rho^\eps(t),s)\\
&=&\sup_\xi \left\{\int  \xi s\ud x-
\frac12 \int  \la \nabla \xi, {B}_\eps^{-1} \nabla \xi \ra \rho^\eps_t \ud x\right\}\\
\Big(&=&
\frac12\int \la\nabla u^\eps, B_\eps^{-1}\nabla u^\eps\ra\rho^\eps_t\ud x
\,\,\,\text{``if there is an $u^\eps$ such that''}\,\,\,
-\nabla\cdot (\rho^\eps_t B_\eps^{-1}\nabla u^\eps) = s\Big),
\end{eqnarray*}
and
\[
g_\infty(t,s) := \psi(\rho(t),s) = 
\sup_\xi \left\{\int  \xi s\ud x-
\frac12 \int  \la \nabla \xi, \bar{B}^{-1} \nabla \xi \ra \rho_t \ud x\right\}.
\]
Again, $g_\eps(t,\cdot)$ and $g_\infty(t,\cdot)$ are convex 
because they are defined as the Legendre transform of $\psi_\eps$ and $\psi$.
By \eqref{psi.lsc0}, \eqref{wk.normal.int} is satisfied. Hence, we have
\[
\int_0^T\psi\big(\rho(t),s(t)\big)\ud t 
\leq
\liminf\int_0^T\psi_\eps\big(\rho^\eps(t), s^\eps(t)\big)\ud t
\]
upon the identification 
$s^\eps(t) = \partial_t\rho^\eps_t$ and 
$s(t) = \partial_t\rho_t$. The fact that $s^\eps\wra s$ in $L^2((0,T);H)$ follows from $\pie\wra \lbar{\pi}$ and $\pt_t f_\eps = \frac{\pt_t \rhoe}{\pie} \lra \pt_t f = \frac{\pt_t \rho_t}{\lbar{\pi}} $
due to Lemma \ref{cor.f}.   Lower bound \eqref{lower3} is thus proved.

Combining Sections \ref{sec4.1}-- \ref{sec4.4}, the proof of Theorem  \ref{thm1} is completed.

\section{Long time behavior}\label{longtime}
In this section, we prove Theorem \ref{thm2}, the long time convergence of the solution $\rhoe$  and energy $E_\eps(\rhoe)$ of
\eqref{epsGF}.
Again, we work under the settings and assumptions introduced in Section \ref{sec:assumptions}. 
The following is our main result.

\begin{thm}\label{thm2}
Let $\rhoe$, $t\geq 0$, be the solution to \eqref{epsFP} with initial data 
{\blue $\rho_0^\eps\in \sP_2(\bR^2)$} satisfying $E_\eps(\rho_0^\eps)<+\8$.
    Suppose $\gamma_W$ and $\delta_W$ are sufficiently small. Then
    \begin{equation}\label{E.Exp.Decay}
    E_\eps(\rhoe)-E_\eps(\mu_\eps) \lesssim e^{-\frac{2\lambda}{\alpha_B}t},
\end{equation}
and
    \begin{equation}\label{Rho.Exp.Decay}
    \|\rhoe- \mu_\eps\|^2_{L^1}   \lesssim e^{-\frac{2\lambda}{\alpha_B}t}.
\end{equation}
In the above, $\lambda$ is the LSI constant for $\pie$ from Lemma \ref{lem:LSI-single}, and 
$\alpha_B$ is from \eqref{B.est}.

Similar estimates also hold in the homogenization limit for $\lbar{E}(\rho_t)$ and $\rho_t$:
\begin{equation}\label{conv.limit}
\lbar{E}(\rho_t)-\lbar{E}(\lbar{\mu}) \lesssim e^{-\frac{2\lambda}{\alpha_{\lbar{B}}}t},
\quad\text{and}\quad
\|\rho_t-\lbar{\mu}\|^2_{L^1}   \lesssim e^{-\frac{2\lambda}{\alpha_{\lbar{B}}}t},
\end{equation}
where $\lbar{E}$ is given in \eqref{0Eng} and 
$\lbar{\mu}:=\text{wk-}\lim_\eps\mu_\eps$ satisfies
\begin{equation}\label{statMFbar}
\lbar{\mu}
= \frac{1}{\lbar{\Xi}}e^{-U_0(x) - W*\lbar{\mu}(x)},\quad\text{with}\quad
\lbar{\Xi}=\int e^{-U_0(x) - W*\lbar{\mu}(x)} \ud x.
\end{equation}
The constant $\alpha_{\lbar{B}}$ bounds $\lbar{B}$: 
$\frac{1}{\alpha_{\lbar{B}}}I \leq \lbar{B} \leq \alpha_{\lbar{B}} I$.
\end{thm}

We point out here that there is some flexibility in terms of the smallness
condition for $\gamma_W$ and $\delta_W$ -- see the final {\bf Step 4} of the proof.
In any case, we work in the regime that 
the interacting term can be treated as a perturbation
to the single particle dynamics.

Before going into the proof, we give some remarks about the above theorem. We basically
use the approach of \cite{Tamura87} which established the result in $L^1$ with finite initial
entropy. An earlier paper \cite{Tamura84} proved a similar result in (weighted-)$L^2$. 
The contents in these works cover our $\eps=1$ case. Our main contribution is the extension of this approach, with uniform estimates to the limit 
$\eps\to0$. For reader's convenience, we briefly describe the results of \cite{Tamura87}. 
Its main conclusion is that under some non-degenerate condition for the energy functional, 
the solution $\rho_t$ converges exponentially fast to an invariant measure $\mu$.
This work in fact
considers a potential function of the form $U(x) \sim |x|^\alpha$ with $\alpha=2$ and
$\alpha >2$. The distinction between the cases is as follows.
\begin{itemize}
\item \cite[Theorem 1.2(i), Theorem 1.3(i)]{Tamura87} 
For $\alpha=2$, with initial data from $\sP_2$, then
\[
\|\rho_t-\mu\|_{L^1} \lesssim e^{-\lambda t}.
\]

\item \cite[Theorem 1.2(ii), Theorem 1.3(ii)]{Tamura87}
For $\alpha>2$, with initial data from $\sP$, then
\[
\|\frac{\rho_t}{\mu}-1\|_\8 \lesssim e^{-\lambda t}.
\]
\end{itemize}
The main difference in topologies follows from the fact that the linearized operator
at $\mu$ is \emph{hyper-contractive} at $\alpha=2$ \cite[Lemma 6.8, 6.10]{Tamura87} 
but is \emph{ultra-contractive} for $\alpha>2$ \cite[Proposition 4.2]{Tamura87}. 
Another comment is that as \cite{Tamura87} worked in a very general setting, the invariant measure might not be unique. Hence his initial data is required to be in some neighborhood of an invariant measure.
But under the uniqueness assumption of invariant meausre, as in our $\gamma_W< 1$ case, such a restriction can be eliminated -- see also \cite[Corollary and Remark in p.465]{Tamura87}.
Finally, for simplicity, we will just consider the case $\alpha=2$ which already contains the 
main ingredients of the proof.

Now we proceed to the proof of Theorem \ref{thm2}. To start, we need to modify the original energy $E_\eps(\rho)$ defined in \eqref{E_mc_eps}. To this end, for 
$p\in\sP(\bR^n)$, define the probability measure
\begin{equation}\label{ep}
    e(p):= \frac{1}{Z_p}e^{-U_\eps - W*p}\in \sP(\bR^n),
    \,\,\,\text{where}\,\,\, Z_p= \int e^{-U_\eps - W*p} \ud x.
\end{equation}
Recall that an invariant measure satisfies $\mu_\eps = e(\mu_\eps)$ -- see \eqref{statMF}. 
With that, we introduce the following energy
\begin{equation}\label{newE}
  \tilde{E}_\eps(\rho):=  \int \Big[U_\eps(x)  +  \log  \rho(x) +  W*\rho  +\log Z_\rho \Big] \rho(x) \ud x. 
\end{equation}
The relation between $\tilde{E}_\eps$ and $E_\eps$ is given by
\begin{equation}\label{Erelation}
    \tilde{E}_\eps(\rho) = E_\eps(\rho) + \int\frac12 (W*\rho)\rho \ud x + \log Z_\rho.
\end{equation}
Note that $\tilde{E}_\eps$ can be recast as
\begin{equation}\label{EngEnt}
\tilde{E}_\eps(\rho)=\int \rho \log \frac{\rho}{e(\rho)} \ud x
\end{equation}
which is exactly the \emph{relative entropy $H(\rho|e(\rho))$} of $\rho$ with respect to $e(\rho)$.
Thus $\tilde{E}_\eps(\rho)$ is nonnegative and vanishes if and only if $\rho=e(\rho)$, i.e., $\rho$ is an invariant measure. For convenience, we recall the definition of $H$,
\begin{equation}
H(p_1|p_2)=\int p_1\log\frac{p_1}{p_2}\ud x,\,\,\,\text{for}\,\,\,p_1,p_2\in\sP(\mathbb R^n).
\end{equation}

Now we proceed to the proof of Theorem \ref{thm2}.

\medskip

\noindent\textbf{Step 1 - LSI for $\tilde{E}_\eps$.}
Using Lemma \ref{lem:LSI-single}, we establish the following LSI for $\tilde{E}_\eps$.
\begin{lem}
    Let $\tilde{E}_\eps$ be defined in \eqref{newE}. Consider the following dissipation functional
    \begin{equation}\label{newD}
        D_\eps(\rho):= \int \left\la \nabla\left(\log \frac{\rho}{\pie} + W\ast\rho\right), \be 
\nabla\left(\log \frac{\rho}{\pie} +W\ast\rho\right)\right\ra \rho\ud x.
    \end{equation}
    Then we have the following LSI:
    \begin{equation}\label{nonlinLSI}
     \tilde{E}_\eps(\rho)=\int \rho \log \frac{\rho}{e(\rho)} \ud x \leq \frac{\alpha_B}{2\lambda} D_\eps(\rho), \quad \forall \rho \in \sP.   
    \end{equation}
    Here, $\alpha_B$ is an upper bound for the matrix $B_\eps$,
    $\lambda=\lambda_0 e^{-C_{W,V}}$ with
    $\lambda_0$ being the Log-Sobolev constant for $\pi_0$ and
     $C_{W,V}:= \max |V_\eps + W*\rho| 
     - \min |V_\eps + W*\rho|\leq C(\|W\|_{L^\8}, \|V_\eps\|_{L^\8})$.
\end{lem}
\begin{proof}
    The idea is to treat the interaction term as a bounded perturbation from $\pi_0$ \eqref{pi0} which satisfies LSI($\lambda_0$). Upon defining the constant
    $$0\leq C_{W,V}:= \max |V_\eps + W*p| - \min |V_\eps + W*p|\leq C(\|W\|_{L^\8} \|V_\eps\|_{L^\8}) < \8,$$
    Lemma \ref{lem:LSI-single} gives that $e(p) = \frac{1}{Z_p}e^{-U_\eps - W*p}\in \sP(\bR^n)$ satisfies LSI($\lambda$) with constant 
    $$\lambda=\lambda_0 e^{-C_{W,V}},$$
    i.e., for any $\eta\in \sP(\bR^n)$,
    $$\int \eta \log \frac{\eta} {e(\rho)} \ud x \leq \frac{1}{2\lambda} \int \eta |\nabla \log \frac{\eta}{e(\rho)}|^2 \ud x.$$
    Inequality \eqref{nonlinLSI} follows by taking $\eta = \rho$:
    \begin{equation}
        \tilde{E}_\eps(\rho)= \int \rho \log \frac{\rho}{e(\rho)} \ud x \leq \frac{1}{2\lambda} \int \rho |\nabla (\log \rho + U_\eps + W*\rho)|^2 \ud x\leq \frac{\alpha_B} {2\lambda} D_\eps(\rho), 
    \end{equation}
    where $\alpha_B$ is an upper bound for the matrix $B_\eps$.
\end{proof}

\bigskip

\noindent\textbf{Step 2 - Decay of energy.}
Note that even though the dissipation functional $D_\eps$ appears in the LSI for
$\tilde{E}_\eps$, it does not exactly corresponds to its decay rate but rather to that for $E_\eps$. For this, we recall \eqref{EDE} here for convenience:
\begin{equation}\label{EDE2}
\frac{d}{dt}E_\eps(\rho^\eps_t) = -D_\eps(\rhoe).
\end{equation}
On the other hand, by exploiting the smallness of $W$ and $\nabla W$ together with the LSI from {\bf Step~1}, we can first establish that $\tilde{E}$ decays to a small neighborhood of zero, as made precise in the lemma right after the following remark.

\begin{rem} In \cite[Theorem 4.2]{Tamura87}, \eqref{EDE2}
is proved only for compactly supported initial data. For general initial data 
$\rho^\eps_0\in\sP_2(\bR^n)$, it is replaced by 
$E_\eps(\rho^\eps_t) - E_\eps(\rho^\eps_s) \leq -\int_s^t D_\eps(\rho^\eps_r)\ud r$,
for all $0<s<t$.
(The modification is caused by the use of approximation and the presence of boundary terms at
$|x|\sim\8$ when applying integration by parts.) 
Upon letting $s\to t$, we in fact have,
\begin{equation}\label{EDE2.ineq}
\frac{d}{dt}E_\eps(\rho^\eps_t)\leq-D_\eps(\rho^\eps_t) 
\end{equation}
which is sufficient for our purpose.
\end{rem}

\begin{lem}\label{Eng.Decay.Lem}
 Let $\rhoe$, $t\geq 0$, be the solution to \eqref{epsFP} and $\tilde{E}_\eps$ and $D_\eps(\rho)$ be defined in \eqref{newE} and \eqref{newD}.  
 Then
 \begin{equation}
     \frac{\ud \tilde{E}_\eps(\rhoe)}{\ud t} \lesssim - \frac12 D_\eps(\rhoe) + \delta_W^2 
     \lesssim -\frac{\lambda}{\alpha_B}\tilde{E}_\eps(\rhoe) + \delta_W^2.
 \end{equation}
Hence,
 \begin{equation}\label{Eng.Decay.Est}
    0 \leq \tilde{E}_\eps(t) 
    \lesssim \frac{\alpha_B}{\lambda}\delta_W^2 + e^{-\frac{\lambda}{\alpha_B}t}\tilde{E}_\eps(0)
\end{equation}
and 
\begin{equation}\label{Eng.Decay.Nbhd}
\tilde{E}_\eps(\rhoe)\lesssim\delta_W^2
\,\,\,\text{for}\,\,\,
t\gtrsim|\log\delta_W|.
\end{equation}
\end{lem}

\begin{proof}
    Starting from the relation \eqref{Erelation} between $E$ and $\tilde{E}$, we compute
    \begin{align}\label{Etm1}
        \frac{\ud \tilde{E}_\eps}{\ud t}  = \frac{\ud  {E}_\eps}{\ud t} +   \int \rhoe (W*\pt_t \rhoe) \ud x- \frac{1}{Z_{\rhoe}} \int e^{-U_\eps - W* \rhoe}(W*\pt_t \rhoe) \ud x.
    \end{align}
    Using \eqref{epsFP}, we have for all $x\in \bR^n$,
    \begin{align*}
        W* \pt_t \rhoe(x) =& \int W(x,y)\nabla \cdot \bbs{\rhoe  \be \nabla (\log \rhoe + U_\eps+W*\rhoe)} \ud y\\
        =& - \int \nabla_y W(x,y)\rhoe  \be \nabla (\log \rhoe + U_\eps+W*\rhoe) \ud y\\
        \leq& \alpha_B \bbs{\int \rhoe |\nabla_y W|^2 \ud y}^{\frac12} D_\eps(\rhoe)^{\frac12}\\
        \leq& \delta_W \alpha_B^\frac12 D_\eps(\rhoe)^\frac12,
    \end{align*}
    recalling that $\delta_W=\|\nabla W\|_{L^\8}$ and $\alpha_B$ bounds the ellipticity of $\be$. 
    Hence $\|W* \pt_t \rhoe\|_\8 \lesssim \delta_W D^\frac12_\eps$. Substituting this into \eqref{Etm1} and making use of \eqref{EDE2.ineq}, we have for some constant $c$ that
    \begin{align*}
        \frac{\ud \tilde{E}_\eps}{\ud t}  \leq  \frac{\ud  {E}_\eps}{\ud t} + 2\delta_W c\sqrt{D_\eps} 
        \leq-D_\eps + 2\delta_W c \sqrt{D_\eps}
        \leq -\frac{1}{2}D_\eps + c\delta_W^2.
    \end{align*}
    Then the LSI \eqref{nonlinLSI}  gives
    \begin{eqnarray}
        \frac{\ud \tilde{E}_\eps}{\ud t} 
        \leq -\frac{\lambda
        }{\alpha_B} \tilde{E}_\eps + c\delta_W^2.
    \end{eqnarray}
    An application of Gronwall's inequality gives \eqref{Eng.Decay.Est} of which \eqref{Eng.Decay.Nbhd} is an immediate consequence.
\end{proof}

\medskip 
\noindent\textbf{Step 3 - Decay in $L^1$.} Here we extend the above energy decay to 
$L^1$-decay. 
Before proceeding, we provide some simple but useful observations about
the map $p\in\sP(\mathbb R^n)\to e(p)$. These are also noted in \cite{Tamura87}.
\begin{enumerate}
\item We have the $L^\8$-$L^1$ bound on $\sP(\mathbb R^n)$,
$\|W*p\|_{L^\infty} \leq \|W\|_{L^\infty}\|p\|_{L^1} = \gamma_W$ and also
\begin{eqnarray}
\|W*p_1-W*p_2\|_{L^\8} \leq \gamma_W\|p_1-p_2\|_{L^1}.
\end{eqnarray}
\item
Similarly, 
\begin{eqnarray}
|e^{-U_\eps(x)-(W*p_1)(x)}-e^{-U_\eps(x)-(W*p_2)(x)}| 
&=&e^{-U_\eps(x)}|e^{-(W*p_1)(x)}-e^{-(W*p_2)(x)}|\nonumber\\
&\lesssim&\gamma_We^{-U_\eps(x)}\|p_1-p_2\|_{L^1}.
\end{eqnarray}
\item 
Note that 
\begin{equation}
Z_p=\int e^{-U_\eps(x)-(W*p)(x)}\ud x
\sim\int e^{-U_\eps(x)}\ud x
\end{equation}
so that $Z_p$ is some order one positive constant. Furthermore,
\begin{eqnarray}
|Z_{p_1}-Z_{p_2}|
&=&|\int e^{-U_\eps(x)-(W*p_1)(x)}\ud x-\int e^{-U_\eps(x)-(W*p_2)(x)}\ud x|\nonumber\\
&=&|\int e^{-U_\eps(x)}(e^{-(W*p_1)(x)}- e^{-(W*p_2)(x)})\ud x|\nonumber\\
&\lesssim&
\gamma_W|\int e^{-U_\eps(x)}\ud x|\|p_1-p_2\|_{L^1}
\lesssim\gamma_W\|p_1-p_2\|_{L^1}.
\end{eqnarray}
\end{enumerate}
Using the above, we compute,
\begin{eqnarray*}
e(p_1)(x)-e(p_2)(x)
&=&
\frac{e^{-U_\eps(x)-(W*p_1)(x)}}{Z_{p_1}}-
\frac{e^{-U_\eps(x)-(W*p_2)(x)}}{Z_{p_2}}\\
&=&
\frac{e^{-U_\eps(x)}}{Z_{p_1}Z_{p_2}}
(e^{-(W*p_1)(x)}Z_{p_2}-e^{-(W*p_2)(x)}Z_{p_1}).
\end{eqnarray*}
Mean value theorem then gives
\begin{eqnarray}
|e(p_1)(x)-e(p_2)(x)|
&\lesssim&\gamma_We^{-U_\eps(x)}\|p_1-p_2\|_{L^1}\nonumber\\
\|e(p_1)-e(p_2)\|_{L^1}
&\lesssim&\gamma_W\|p_1-p_2\|_{L^1}\label{e.contract}
\end{eqnarray}
and
\begin{eqnarray}
\frac{e(p_1)}{e(p_2)}-1
&=&\frac{Z_{p_2}}{Z_{p_1}}e^{-(W*p_1)+(W*p_2)}-1=\frac{Z_{p_2}}{Z_{p_1}}(e^{-(W*p_1)+(W*p_2)}-1)+\frac{Z_{p_2}}{Z_{p_1}}-1\nonumber\\
\|\frac{e(p_1)}{e(p_2)}-1\|_{L^\8}
&\lesssim&\gamma_W\|p_1-p_2\|_{L^1}\label{eL8L1}.
\end{eqnarray}

\begin{lem}\label{lem5.5}
Let $\tilde{E}$ be defined in \eqref{newE}. For
$\gamma_W$ small enough and for $t\gtrsim |\log\delta_W|$, then it holds that,
\begin{equation}
 \|\rhoe-\mu_\eps\|_{L^1}\lesssim \delta_W.   
\end{equation}
\end{lem}
\begin{proof}
    By Pinsker's inequality -- see also \cite[Lemma 6.1]{Tamura87}, we have for all $\rho\in\sP(\mathbb R^n)$,
\begin{equation}
    \|\rho - e(\rho)\|_{L^1}^2 \leq 2\tilde{E}_\eps(\rho)=2\int \rho \log \frac{\rho}{e(\rho)} \ud x.
    \label{pinsker1}
\end{equation}
Since $\mu_\eps=e(\mu_\eps)$, the triangle inequality then gives
\begin{align*}
    \|\rho-\mu_\eps\|_{L^1}\leq \|\rho-e(\rho)\|_{L^1} + \|e(\rho)-\mu_\eps\|_{L^1}=\|\rho-e(\rho)\|_{L^1} + \|e(\rho)-e(\mu_\eps)\|_{L^1}.
\end{align*}
By \eqref{e.contract}, $e(\cdot)$ is a contraction for $\gamma_W$ small enough
-- see also \cite[Theorem 4.1(iii)]{Tamura84}. Therefore,
\begin{eqnarray}
     \|\rho-\mu_\eps\|_{L^1}\lesssim \|\rho-e(\rho)\|_{L^1}\leq
     \sqrt{2\tilde{E}_\eps(\rho)}.
     \label{pinsker2}
\end{eqnarray}
The conclusion then follows from Lemma \ref{Eng.Decay.Lem} upon setting $\rho=\rhoe$.
\end{proof}

\noindent\textbf{Step 4 - Exponential convergence to $\mu_\eps$.}
This is the final step which combines all the previous estimates.
We do make use of an extra key technical statement \cite[Lemma 6.7]{Tamura87} which says that
\begin{equation}\label{TamureLem6.7}
\text{if}\,\,\,\|\frac{e(\rho)}{\mu_\eps}-1\|_{L^\8}\ll 1,
\,\,\,\text{then}\,\,\,
    E_\eps(\rho)-E_\eps(\mu_\eps) \lesssim  \tilde{E}_\eps(\rho).
\end{equation}
The above allows us to infer the temporal decay property of ${E}_\eps(\rho)$ from that
of $\tilde{E}_\eps(\rho)$. (We will provide an outline of the proof of this result in Appendix \ref{PfTLem6.7} as it involves several non-trivial steps.)
The hypothesis in \eqref{TamureLem6.7} can be realized as follows. Setting $p_1=\rhoe$ and $p_2=\mu_\eps$
in \eqref{eL8L1} gives
\begin{equation}\label{OpenNbhd1}
\|\frac{e(\rhoe)}{\mu_\eps}-1\|_{L^\8}\lesssim \gamma_W\|\rhoe-\mu_\eps\|_{L^1}.
\end{equation}
By Lemma \ref{lem5.5}, we have for $\gamma_W,\,\,\delta_W$ small enough and $t\gtrsim|\log\delta_W|$
that,
\begin{equation}\label{OpenNbhd2}
\|\frac{e(\rhoe)}{\mu_\eps}-1\|_{L^\8}\lesssim \gamma_W\delta_W\ll1.
\end{equation}
This is the place where we see the flexibility of smallness of $\gamma_W$ and $\delta_W$
-- all is needed is $\gamma_W\delta_W\ll 1$. The initial assumption $\gamma_W<1$ is to ensure that the invariant measure is unique -- Corollary \ref{cor:E}.

With that, we are ready to finish the proof of Theorem \ref{thm2}. By the energy dissipation equality
\eqref{EDE2} and the LSI \eqref{nonlinLSI}, we obtain
\begin{equation}
 \begin{aligned}
    \frac{\ud }{\ud t}(E_\eps(\rhoe) - E_\eps(\mu_\eps)) = -D_\eps(\rhoe) \leq -\frac{2\lambda}{\alpha_B} \tilde{E}_\eps(\rhoe) \lesssim -\frac{2\lambda}{\alpha_B}(E_\eps(\rhoe)-E_\eps(\mu_\eps) ).
\end{aligned}   
\end{equation}
This yields the exponential decay \eqref{E.Exp.Decay} of $E_\eps(\rhoe)$ to $E_\eps(\mu_\eps)$:
$$E_\eps(\rhoe)-E_\eps(\mu_\eps) \lesssim e^{-\frac{2\lambda}{\alpha_B}t}.$$

For the $L^1$-decay \eqref{Rho.Exp.Decay} of $\rhoe$, we can start from \eqref{Erelation}
to arrive at
\begin{align*}
    \int \rhoe\log\frac{\rhoe}{\mu_\eps}\ud x = E_\eps(\rhoe)-E_\eps(\mu_\eps) + O(\gamma_W\|\rhoe - \mu_\eps\|^2_{L^1}).
\end{align*}
(More precisely, the following holds for all $\rho\in\sP(\mathbb R^n)$ -- see \cite[Lem 6.2(6.3)]{Tamura87},
\[
\int \rho\log\frac{\rho}{\mu_\eps}\ud x 
= E_\eps(\rho) - E_\eps(\mu_\eps) - \frac12\int W*(\rho-\mu_\eps)(\rho-\mu_\eps)\ud x.)
\]
Then Pinsker's inequality \eqref{pinsker1} gives
\begin{equation}
\|\rhoe-\mu_\eps\|^2_{L^1}
\lesssim
\int \rhoe\log\frac{\rhoe}{\mu_\eps}\ud x 
= E_\eps(\rhoe)-E_\eps(\mu_\eps) + O(\gamma_W\|\rhoe - \mu_\eps\|^2_{L^1}).
\end{equation}
Hence for $\gamma_W$ small enough, we conclude that,
\begin{equation}
    \|\rhoe - \mu_\eps\|^2_{L^1} \lesssim E(\rhoe)-E(\mu_\eps) \lesssim e^{-\frac{2\lambda}{\alpha_B}t}.
\end{equation}

The proof of \eqref{conv.limit} follows from exactly the same with steps.

We then have completed the proof of Theorem \ref{thm2}.

\appendix
\section{Proof (Outline) of \eqref{TamureLem6.7} \cite[Lemma 6.7]{Tamura87}}\label{PfTLem6.7}
We essentially follow 
but simplify somewhat \cite[Lemma 6.3--6.7]{Tamura87}, tailored to our setting.

\begin{itemize}
\item \cite[Lemma 6.3]{Tamura87}
By direct calculation \cite[Lemma 6.2(6.3)]{Tamura87}, for any $p,\,\,q\in\sP(\mathbb R^n)$, it holds that
\begin{multline*}
E_\eps(p) - E_\eps(q)
=H(p|e(q))
+ \frac12\int W*(p-q)(p-q)\ud x
+ \frac12\int W*(q-e(q))(q-e(q))\ud x.
\end{multline*}
Taking $q=\mu_\eps$ and using Pinsker's Inequality \eqref{pinsker1} give that 
\begin{equation}\label{T6.3}
E_\eps(p)-E_\eps(\mu_\eps)
\leq (1+\gamma_W)\big(H(p|e(p))+H(e(p)|\mu_\eps\big).
\end{equation}

\item \cite[Lemma 6.5]{Tamura87} states that there exist positive constants $\gamma_1$ and $\delta$ so that
\begin{equation}\label{T6.5}
H(p|\mu_\eps) \leq \gamma_1 H(p|e(p))\,\,\,\text{holds
for any $p\in\sP(\mathbb R^n)$ with $\|\frac{p}{\mu_\eps}-1\|_{L^\8}< \delta$.}
\end{equation}
The above essentially reflects the fact that $H(\cdot|\mu_\eps) \sim H(\cdot|e(p))$ 
when $p$ is appropriately close to $\mu_\eps$.
The proof of this step involves the linearization of the entropy functional $H(\cdot|\cdot)$ near $\mu_\eps$. More precisely, we have
\begin{eqnarray}
\left|
H(\mu_\eps+h|\mu_\eps)-\frac12\|h\|^2_{\mu_\eps^{-1},0}
\right| 
&\lesssim& \left\|\frac{h}{\mu_\eps}\right\|_{L^\8}\|h\|^2_{\mu_\eps^{-1},0},\label{lin.H.1}\\
\left|
H(\mu_\eps+h|e(\mu_\eps+h))-\frac12\|(I-K_{\mu_\eps})h\|^2_{\mu_\eps^{-1},0}
\right| 
&\lesssim& \left\|\frac{h}{\mu_\eps}\right\|_{L^\8}\|h\|^2_{\mu_\eps^{-1},0}\label{lin.H.2}
\end{eqnarray}
where 
\begin{eqnarray}
K_{\mu_\eps}h(x)&=&-\mu_\eps(x)\big[W*h(x) - \langle W*h,\mu_\eps\rangle_{L^2}\big],\\
L^2_{\mu_\eps^{-1},0}&=&
\left\{
g:\mathbb R^n\longrightarrow\mathbb R:\,\,\,
\int g\ud x=0,\,\,\,\int \frac{g^2}{\mu_\eps} \ud x< \8
\right\}.
\end{eqnarray}
Formally, we have $I-K_{\mu_\eps}=D^2E_\eps(\mu_\eps)$ on $L^2_{\mu_\eps,0}$ and 
$(I-K_{\mu_\eps})^{-1}$ exists on $L^2_{\mu_\eps,0}$ for small enough $\gamma_W$.

Now we estimate,
\begin{eqnarray*}
&&H(\mu_\eps+h|\mu_\eps)\\
&\lesssim_{\text{\eqref{lin.H.1}}} &
\|h\|^2_{\mu_\eps^{-1},0}+\left\|\frac{h}{\mu_\eps}\right\|_{L^\8}\|h\|^2_{\mu_\eps^{-1},0}\\
&\lesssim&
\|(I-K_{\mu_\eps})^{-1}\|^2\|(I-K_{\mu_\eps})h\|^2_{\mu_\eps^{-1},0}
+\left\|\frac{h}{\mu_\eps}\right\|_{L^\8}\|h\|^2_{\mu_\eps^{-1},0}\\
&\lesssim_{\text{\eqref{lin.H.2}}}&
\|(I-K_{\mu_\eps})^{-1}\|^2\left(
H(\mu_\eps+h|e(\mu_\eps+h))+\left\|\frac{h}{\mu_\eps}\right\|_{L^\8}\|h\|^2_{\mu_\eps^{-1},0}
\right)
+\left\|\frac{h}{\mu_\eps}\right\|_{L^\8}\|h\|^2_{\mu_\eps^{-1},0}\\
&\lesssim&
\|(I-K_{\mu_\eps})^{-1}\|^2
H(\mu_\eps+h|e(\mu_\eps+h))+
(1+\|(I-K_{\mu_\eps})^{-1}\|^2)\left\|\frac{h}{\mu_\eps}\right\|_{L^\8}\|h\|^2_{\mu_\eps^{-1},0}
\\
&\lesssim_{\text{\eqref{lin.H.1}}}&
\|(I-K_{\mu_\eps})^{-1}\|^2
H(\mu_\eps+h|e(\mu_\eps+h))+
(1+\|(I-K_{\mu_\eps})^{-1}\|^2)\left\|\frac{h}{\mu_\eps}\right\|_{L^\8}
\frac{H(\mu_\eps+h|\mu_\eps)}{1-C\left\|\frac{h}{\mu_\eps}\right\|_{L^\8}}
\end{eqnarray*}
for some constant $C$.
Hence for $\left\|\frac{h}{\mu_\eps}\right\|_{L^\8}\ll1$, we have
\[
H(\mu_\eps+h|\mu_\eps)\lesssim H(\mu_\eps+h|e(\mu_\eps+h))
\]
giving \eqref{T6.5}.

\item \cite[Lemma 6.6]{Tamura87} Fairly direct calculation and estimation gives that there is a positive constant $\gamma_2$ such that for any $p\in\sP(\mathbb R^n)$, we have
\begin{equation}\label{T6.6}
H(e(p)|e(e(p)))\leq\gamma_2 H(p|e(p)).
\end{equation}
\end{itemize}

With the above, the result follows from the following chain of estimates:
\begin{align*}
E_\eps(p)-E_\eps(\mu_\eps)
&\leq_{\text{\eqref{T6.3}}} (1+\gamma_W)\big(H(p|e(p))+H(e(p)|\mu_\eps)\big)\\
&\leq_{\text{\eqref{T6.5}}} (1+\gamma_W)\big(H(p|e(p))+\gamma_1 H(e(p)|e(e(p))\big)\\
&\leq_{\text{\eqref{T6.6}}} (1+\gamma_W)(1+\gamma_1\gamma_2)H(p|e(p))
\end{align*}
which is exactly \eqref{TamureLem6.7}.
Note that when applying \eqref{T6.5} in the above, we are in fact relying on
$\|\frac{e(p)}{\mu_\eps}-1\|_{L^\8}\ll1$ which coincides with the 
hypothesis in \eqref{TamureLem6.7}. That this can be realized is already explained in
\eqref{OpenNbhd1}-\eqref{OpenNbhd2}.

\section{Asymptotic analysis for the $\eps$-gradient flow}\label{asym.exp}
In this section, we use the method of asymptotic expansion to analyze
the convergence of the $\eps$-Fokker-Planck equation \eqref{epsFP} 
(or \eqref{epsGF}) to the limiting homogenized one \eqref{0GF}.

We recall the governing equation for $\rhoe$ and its version in terms of $f^\eps$ here for convenience:
\begin{equation}
\pt_t \rhoe = \nabla \cdot \bbs{\pi_\eps B_\eps^{-1} \nabla \frac{\rhoe}{\pi_\eps} + \rhoe B_\eps^{-1}\nabla W*\rhoe   },
\end{equation}
\begin{align}
\pt_t f^\eps_t = \frac{1}{\pi_\eps} \nabla \cdot \bbs{   \pi_\eps B_\eps^{-1}  \nabla f^{\eps}_t + f^\eps_t   \pi_\eps B_\eps^{-1}    \nabla W*(f^\eps_t\pie)  }.
\end{align}

Recall the assumptions \eqref{Bform} and \eqref{pie}
for $B_\eps$ and $\pie$ in Section \ref{main} and the definition of
fast variable $\displaystyle y= \frac{x}{\eps}$. 
Introducing,
\begin{eqnarray}\label{A.def}
A(x,y) &=& \pi(x,y)B^{-1}(y),\\
I_\eps(\varphi, x) 
&=& \nabla W\ast(\varphi\pi(x,\frac{x}{\eps}))
= \int \nabla_x W(x,z)\varphi(z)\pi(z,\frac{z}{\eps})\,dz,
\end{eqnarray}
then \eqref{backward} reads
\begin{equation}\label{tmf}
\pt_t f^\eps =\frac{1}{\pie} \nabla \cdot \Big(A(x,\frac{x}{\eps}) \nabla f^\eps + f^\eps A(x,\frac{x}{\eps})I_\eps(f^\eps, x)\Big).
\end{equation}
Note that for all $m\geq 0$
\[
\Big|D^m_xI_\eps(\varphi, x)\Big|
\lesssim
\int \left|D^m_x\nabla_x W(x,z)\right|\varphi(z)\pi(z,\frac{z}{\eps})\,dz
\leq O(1)
\]
so that to leading order $I_\eps(\varphi, x)$ contains no $y$ components. Hence without loss of 
generality, we can replace $I_\eps(\varphi, x)$ by the following limit
\[
\lim_{\eps\to0}I_\eps(\varphi, x) = \int\nabla_x W(x,z)\varphi(z)\langle\pi\rangle(z)\ud z
=: \overline{I}(\varphi,x).
\]
We also introduce:
\[
\overline{I}_i(\varphi,x):=\int\partial_{x_i} W(x,z)\varphi(z)\langle\pi\rangle(z)\ud z.
\]
 
Now consider the ansatz
\begin{equation}
f^\eps(x,t) = f_0(x,\frac{x}{\eps},t) + \eps f_1(x,\frac{x}{\eps},t) + \eps^2 f_2(x,\frac{x}{\eps},t)+O(\eps^3),
\end{equation}
where $f_i=f_i(x,y)$, $i=0,1,2$, are $1$-periodic in $y$.
Substituting it into \eqref{tmf}, we have
\begin{eqnarray}
&&\pt_t 
\big(f_0+\eps f_1 + \eps^2 f_2 + \cdots\big)\nonumber\\
&=& \frac{1}{\pi(x,y)}\left(\nabla_x + \frac{1}{\eps} \nabla_y\right)\cdot\\
&& \Big[A(x,y)\left(\nabla_x+\frac{1}{\eps}\nabla y\right)
\big(f_0+\eps f_1 + \eps^2 f_2 + \cdots\big)\nonumber\\
&&
+\big(f_0+\eps f_1 + \eps^2 f_2 + \cdots\big)A(x,y)I(f_0+\eps f_1 + \eps^2 f_2 + \cdots, x)
\Big].\label{feps}
\end{eqnarray} 
Terms of different orders are analyzed as follows.

\begin{description}
\item[(I) $\frac{1}{\eps^2}$-terms] They satisfy,
\begin{align*}
\nabla_y \cdot \Big(A(x,y) \nabla_y f_0(x,y,t)\Big) =0.
\end{align*}
Multiply the above by $f_0(x,y,t)$ and then integrate over $y$ gives
$\displaystyle \int |\nabla_y f_0(x,y,t)|^2\ud y = 0$ 
which implies $f_0(x,y,t)=f_0(x,t).$

\item[(II) $\frac{1}{\eps}$-terms] They satisfy,
\begin{eqnarray}
\nabla_y \cdot \Big(A(x,y) \nabla_y f_1\Big)
&=& -\nabla_y\cdot\Big[A(x,y)\nabla_x f_0(x)\Big] - \nabla_y\cdot\Big[f_0(x)A(x,y)\bar{I}(f_0,x)\Big]\\
&=& -\nabla_y\cdot\Big[A(x,y)\nabla_x f_0(x)\Big] - f_0(x)\nabla_y\cdot\Big[A(x,y)\bar{I}(f_0,x)\Big].
\end{eqnarray} 
For $i=1,2,\ldots d$, let $w_i(y)$ be the solution to the cell problem
\begin{equation}\label{cell1}
\nabla_y \cdot \Big(A(x,y)\nabla_y w_i(x,y)\Big) +  \nabla_y \cdot \Big(A(x,y) \vec{e}_i\Big)=0,
\end{equation}
where $\vec{e}_i$ is the unit vector in $i$-coordinate.
The above equation is solvable for each $i$ due to the compatibility condition
$$
\int \nabla_y \cdot \Big(A(x,y) \vec{e}_i\Big) \ud y = 0.
$$
Then we have
\begin{equation}
f_1(x,y) = \sum_j \big(\partial_{x_j}f_0(x)\big)w_j(x,y) 
+ f_0(x)\sum_j w_j(x,y)\bar{I}_j(f_0,x)
\end{equation} 
where $w = (w_1, w_2, \ldots, w_d)^T$.

\item[(III) $O(1)$-terms] Collecting the $O(1)$-terms in \eqref{feps} and
integrating with respect to $y$ lead to
\begin{multline*}
\partial_t f_0(x,t) =
\frac{1}{\pi(x,y)}\Big\{
\nabla_y\cdot\big[A(x,y)\nabla_y f_2\big]
+\nabla_x\cdot\big[A(x,y)\nabla_y f_1\big]
+\nabla_y\cdot\big[A(x,y)\nabla_x f_1\big]
\\
+\nabla_x\cdot\big[A(x,y)\nabla_x f_0\big]
+\nabla_x\cdot\big[f_0(x) A(x,y)I_\eps(x)\big]
+\nabla_y\cdot\big[f_1(x,y)A(x,y)I_\eps(x)\big]
\Big\}.
\end{multline*}
Upon integrating in $y$, we get
\begin{align*}
\pt_t f_0(x,t)  
= \frac{1}{\langle{\pi}\rangle}\nabla_x\cdot\Big\{\bar{A}(x)\nabla f_0 + f_0\bar{A}(x)\bar{I}(f_0,x)\Big\}
\end{align*}
where
\begin{equation}\label{eff.A}
\bar{A}_{ij}(x) = \int_{y\in\mathbb{T}^n}
\big(A_{ik}(x,y)\partial_{y_k}w_j(x,y) + A_{ij}(x,y)\big)\ud y.
\end{equation}
\end{description}

With the above, we can now transform the dynamics of $f_0$ as follows:
\begin{eqnarray*}
\partial_t f_0 = \frac{1}{\langle{\pi}\rangle}\nabla\cdot\left\{
f_0\left(\bar{A}\frac{\nabla f_0}{f_0} + \bar{A}\bar{I}(f_0)\right)
\right\}
\end{eqnarray*}
In terms of $\bar{\rho}:=f_0 \lbar{\pi}$, the above then becomes
\begin{eqnarray*}
\partial_t \bar{\rho} = \nabla\cdot\left\{
\bar{\rho}\left(\frac{\bar{A}}{\lbar{\pi}}\nabla\log{\frac{\bar{\rho}}{\lbar{\pi}}} + \frac{\bar{A}}{\lbar{\pi}}\bar{I}\left(\frac{\bar{\rho}}{\lbar{\pi}}\right)\right)
\right\}
\end{eqnarray*}
where now
\[
\bar{I}\left(\frac{\bar{\rho}}{\lbar{\pi}},x\right)
=\int\nabla_x W(x,z)\bar{\rho}(z)\ud z
=\nabla_x W\ast\bar{\rho}
\]
so that
\begin{equation}\label{eff.rho.eqn}
\partial_t \bar{\rho} = \nabla\cdot\left\{
\bar{\rho}\frac{\bar{A}}{\lbar{\pi}}\nabla\left(\log{\frac{\bar{\rho}}{\lbar{\pi}}} + W\ast\bar{\rho}\right)
\right\}.
\end{equation}
  Upon introducing $\displaystyle\lbar{B}^{-1}=\frac{\bar{A}}{\lbar{\pi}}$, the above becomes
\begin{equation}\label{eff.rho.eqn.2}
\partial_t \bar{\rho} = \nabla\cdot\left\{
\bar{\rho}\lbar{B}^{-1}\nabla\left(\log{\frac{\bar{\rho}}{\lbar{\pi}}} + W\ast\bar{\rho}\right)
\right\}.
\end{equation}
leading to the form \eqref{0GF}.

Note that in general, for $\pi=\pi(x,y)$, the formula for $\bar{A}$ involves  interaction between $\pi$ and $B$. On the other hand, for $\pi=\pi(x)$ with no dependence in the fast
variables, or for the simpler uniform convergence
case $\pie=\pie^{\text{II}}$ in \eqref{pie} which converges uniformly to 
$\pi_0$, such an interaction disappears. More precisely, for $A(x,y) = \pi(x)B^{-1}(y)$, 
the cell problem \eqref{cell1} now becomes
\begin{equation}
\nabla_y \cdot \Big(B^{-1}(y)\nabla_y w_i(y)\Big) +  \nabla_y \cdot \Big(B^{-1}(y) \vec{e}_i\Big)=0
\end{equation}
which does not involve $\pi$ at all. Then the formula \eqref{eff.A} for the effective coefficient becomes
\begin{equation}
\bar{A}_{ij}(x) = \pi(x)\tilde{A}_{ij},\quad\text{where}\quad
\tilde{A}_{ij}=\int_{y\in\mathbb{T}^n}
\Big((B^{-1})_{ik}(x,y)\partial_{y_k}w_j(x,y) + (B^{-1})_{ij}(x,y)\Big)\ud y.
\end{equation}
Then equation \eqref{eff.rho.eqn} becomes
\begin{equation}\label{eff.rho.eqn0}
\partial_t \bar{\rho} = \nabla\cdot\left\{
\bar{\rho}\tilde{A}\nabla\left(\log{\frac{\bar{\rho}}{\lbar{\pi}}} +  W\ast\bar{\rho}\right)
\right\}.
\end{equation}

\bigskip

\noindent
{\bf Acknowledgements.}
The research of YG is partially supported by NSF CAREER Award DMS-2440651.

\medskip

\noindent{\bf Data Availability.} No data have been used or created for this work.

\medskip

\noindent
{\bf Conflict of interest.} The authors declare no conflict of interest.

\bibliographystyle{alpha}
\bibliography{homo_bib}

\end{document}